\documentclass[11pt]{amsart}

\usepackage{graphicx}
\usepackage{epsfig}
\usepackage{amsmath}
\usepackage{amssymb}
\usepackage{bigints}
\usepackage{tikz}
\usepackage{graphicx}
\usepackage{epsfig}
\usepackage[noadjust]{cite}
\usepackage{xcolor}

\usepackage[colorlinks,citecolor=red,urlcolor=blue,bookmarks=false,hypertexnames=true]{hyperref}

\theoremstyle{plain}
\newtheorem{theorem}                {Theorem}      [section]
\newtheorem*{theorem*}                {Theorem \ref{thm:appl}}
\newtheorem{proposition}  [theorem]  {Proposition}

\newtheorem{lemma}        [theorem]  {Lemma}

\theoremstyle{definition}

\newtheorem{remark}       [theorem]  {Remark}

\DeclareMathOperator{\trace}{trace} 

\DeclareMathOperator{\grad}{grad}

\numberwithin{equation}{section}

\begin{document}

\title[Complete biconservative surfaces in the hyperbolic space $\mathbb{H}^3$]
{Complete biconservative surfaces in the hyperbolic space $\mathbb{H}^3$}

\author{Simona~Nistor, Cezar~Oniciuc}

\address{Faculty of Mathematics\\ Al. I. Cuza University of Iasi\\
Bd. Carol I, 11 \\ 700506 Iasi, Romania} \email{nistor.simona@ymail.com}

\address{Faculty of Mathematics\\ Al. I. Cuza University of Iasi\\
Bd. Carol I, 11 \\ 700506 Iasi, Romania} \email{oniciucc@uaic.ro}

\thanks{This work was supported by a grant of the ``Alexandru Ioan Cuza'' University of Iasi, within the Research Grants program, Grant UAIC, code GI-UAIC-2018-04}



\begin{abstract}
We construct simply connected, complete, non-$CMC$ biconservative surfaces in the $3$-dimensional hyperbolic space $\mathbb{H}^3$ in an intrinsic and extrinsic way. We obtain three families of such surfaces, and, for each surface, the set of points where the gradient of the mean curvature function does not vanish is dense and has two connected components. In the intrinsic approach, we first construct a simply connected, complete abstract surface and then prove that it admits a unique biconservative immersion in $\mathbb{H}^3$. Working extrinsically, we use the images of the explicit parametric equations and a gluing process to obtain our surfaces. They are made up of circles (or hyperbolas, or parabolas, respectively) which lie in $2$-affine parallel planes and touch a certain curve in a totally geodesic hyperbolic surface $\mathbb{H}^2$ in $\mathbb{H}^3$.
\end{abstract}

\maketitle
\section{Introduction}

In the last years the theory of \textit{biconservative submanifolds} proved to be a very interesting research topic (see, for example, \cite{CMOP14,FOP15,F15,FT16,MOR16,S12,S15,T15,UT16,YT18}). Since, in certain geometric contexts, finding \textit{biharmonic submanifolds} is difficult, the interest in biconservative submanifolds, which generalize the biharmonic ones, has appeared naturally.

The \textit{biharmonic maps} between two Riemannian manifolds $\left(M^m,g\right)$ and $\left(N^n,h\right)$ are characterized by the vanishing of the associated \textit{bitension field}
$$
\tau_{2}(\varphi)=-\Delta^{\varphi}\tau(\varphi)-\trace_g R^N(d\varphi,\tau(\varphi))d\varphi,
$$
and are critical points of the \textit{bienergy functional} (see \cite{J86}).

When $\varphi:\left(M^m,g\right)\to\left(N^n,h\right)$ is an isometric immersion, i.e., $M$ is a submanifold of $(N,h)$, and $\varphi$ is a biharmonic map, we say that $M$ is a \textit{biharmonic submanifold}. In this case, the biharmonic equation $\tau_{2}(\varphi)=0$ splits into the tangent and normal part (see \cite{BMO13,LMO08,O02,O10}). Submanifolds with $\left(\tau_{2}(\varphi)\right)^\top=0$ are called \textit{biconservative submanifolds}.

We note that submanifolds with divergence-free stress bienergy tensor are precisely the biconservative submanifolds (see \cite{J87,LMO08}).

The biconservative submanifolds were studied for the first time in 1995 by Th. Hasanis and Th. Vlachos (see \cite{HV95}). In that paper the biconservative hypersurfaces in the Euclidean space $\mathbb{R}^n$ were called \textit{H-hypersurfaces} and were fully classified in $\mathbb{R}^3$ and $\mathbb{R}^4$.

When the ambient space is a 3-dimensional space form $N^3(c)$,  i.e., a $3$-dimensional real space with constant sectional curvature $c$, it is easy to see that surfaces with constant mean curvature (\textit{$CMC$ surfaces}) are biconservative.  Indeed, a surface $\varphi:M^2\rightarrow N^3(c)$ is biconservative if and only if
\begin{equation}\label{eq2}
A(\grad f)=-\frac{f}{2}\grad f,
\end{equation}
where $A$ is the shape operator of $M$ and $f=\trace A$ is its mean curvature function.

Therefore, we are interested in biconservative surfaces which \textit{are non-$CMC$}, i.e.,$\grad f\neq 0$ at any point of an open subset of $M$.

The explicit local parametric equations of biconservative surfaces in $\mathbb{R}^3$, $\mathbb{S}^3$ and $\mathbb{H}^3$ were determined in \cite{CMOP14} and \cite{F15}. We mention that, when the ambient space is $\mathbb{R}^3$, the result in \cite{HV95} was reobtained in \cite{CMOP14}. Also, some global and uniqueness results concerning biconservative surfaces in $\mathbb{R}^3$ and $\mathbb{S}^3$ are given in \cite{NPhD17,N16,NO19, NO17}.

The aim of this paper is to obtain global results concerning non-$CMC$ biconservative surfaces in the hyperbolic space $\mathbb{H}^3$. We start with a short section where we recall some known properties of biconservative surfaces in $N^3(c)$ with a nowhere vanishing gradient of the mean curvature function. Then, in Section \ref{sec-Intrinsic}, working in an intrinsic way, we first construct a certain simply connected, complete abstract surface gluing by symmetry along their common boundary two abstract standard biconservative surfaces (see Theorem \ref{th-metricComplete}). These abstract standard biconservative surfaces were first determined in \cite{NO17} but, in order to perform the gluing process, we change the coordinates to write the metric in the most appropriate form for our purpose. Then, we prove (in Theorem \ref{th:EUH3}), that the above simply connected, complete abstract surface admits a unique biconservative immersion in $\mathbb{H}^3$. Moreover, for this immersion, $\grad f$ is different from zero on a dense set. We end the section by stating two conjectures. The first one claims the uniqueness of simply connected, complete, non-$CMC$ biconservative surfaces in $N^3(c)$, and the second one says that any compact biconservative surface in $N^3(c)$ is $CMC$.

In Section \ref{sec-Extrinsic} we basically reobtain Theorem \ref{th:EUH3} by constructing complete, non-$CMC$ biconservative surfaces in $\mathbb{H}^3$. This construction is done using the images of the explicit parametric equations and a gluing process (see Theorem \ref{th-complete-extrinsic}). More precisely, we begin with the known one-parameter family of standard biconservative surfaces with a nowhere vanishing $\grad f$, indexed by a real constant, whose explicit parametric equations were given in \cite{CMOP14,F15}. Then, according to the sign of that constant, we obtain three families of complete, non-$CMC$ biconservative surfaces in $\mathbb{H}^3$.

\textbf{Conventions.} We assume that all manifolds are connected and use the following sign conventions for the rough Laplacian acting on sections of $\varphi^{-1}(TN)$ and for the curvature tensor field of $N$, respectively:
$$
\Delta^{\varphi}=-\trace_{g} \left(\nabla^{\varphi}\nabla^{\varphi}-\nabla^{\varphi}_{\nabla}\right)
$$
and
$$
R^N(X,Y)Z=[\nabla^N_X,\nabla^N_Y]Z-\nabla^N_{[X,Y]}Z.
$$

\textbf{Acknowledgments.} We would like to thank to  D.~Fetcu and S.~Moroianu for useful discussions and suggestions.

\section{Preliminaries} \label{sec-Preliminarii}

\vspace{0.2cm}
For the sake of completeness, we present some known results concerning biconservative surfaces in three-dimensional space forms $N^3(c)$, that will be useful in the following sections.

First, we recall some properties of biconservative surfaces in $N^3(c)$ with a nowhere vanishing $\grad f$.

\begin{theorem}[\cite{CMOP14}]\label{thm:CMOP}
Let $\varphi:M^2\to N^3(c)$ be a biconservative surface with $\grad f\neq 0$ at any point of $M$. Then the Gaussian curvature $K$ satisfies
\begin{itemize}
\item [(i)]
    $$
    K=\det A+c=-\frac{3f^2}{4}+c;
    $$
\item [(ii)] $f^2>0$, i.e., $c-K>0$, $\grad K\neq 0$ on $M$, and the level curves of $K$ are circles in $M$ with constant curvature
    $$
    \kappa=\frac{3|\grad K|}{8(c-K)};
    $$
\item [(iii)]
\begin{equation}\label{eq:similarRicci}
(c-K)\Delta K-|\grad K|^2-\frac{8}{3}K(c-K)^2=0,
\end{equation}
where $\Delta$ is the Laplace-Beltrami operator on $M$.
\end{itemize}
\end{theorem}

In particular, it follows that $M$ is orientable and choosing $H/|H|$ as the unit normal vector field, we have $f>0$.

From now on, we will assume that any abstract surface is oriented.

\begin{remark}
From \eqref{eq:similarRicci} we can see that the biconservative surfaces are closely related to the Ricci surfaces (see \cite{MM15}) and the link between them was studied in \cite{FNO16}.
\end{remark}

Next, we present the characterization theorem and an existence and uniqueness result concerning biconservative surfaces in $N^3(c)$.

\begin{theorem}[\cite{FNO16}]\label{thm:char}
Let $\left(M^2,g\right)$ be an abstract surface. Then $M$ can be locally isometrically embedded in a space form $N^3(c)$ as a biconservative surface with the gradient of the mean curvature different from zero everywhere if and only if the Gaussian curvature $K$ satisfies $c-K(p)>0$, $(\grad K)(p)\neq 0$, for any $p\in M$, and its level curves are circles in $M$ with constant curvature
$$
\kappa=\frac{3|\grad K|}{8(c-K)},
$$
where $c\in\mathbb{R}$ is a fixed constant.
\end{theorem}

\begin{theorem}[\cite{FNO16,NPhD17}]\label{thm:reformulate}
Let $\left(M^2,g\right)$ be an abstract surface and $c\in\mathbb{R}$ an arbitrarily fixed constant. Assume that $c-K>0$ and $\grad K\neq0$ at any point of $M$, and the level curves of $K$ are circles in $M$ with constant curvature
$$
\kappa=\frac{3|\grad K|}{8(c-K)}.
$$
Then, locally, there exists a unique biconservative embedding $\varphi:\left(M^2,g\right)\to N^3(c)$. Moreover, the mean curvature function is positive and its gradient is different from zero at any point.
\end{theorem}

Next, we give some equivalent conditions with the hypothesis from the above theorem.

\begin{theorem}[\cite{FNO16,N16,NO17}]\label{thm:carac}
Let $\left(M^2,g\right)$ be an abstract surface with Gaussian curvature $K$ satisfying $c-K(p)>0$ and $(\grad K)(p)\neq 0$ at any point $p\in M$, where $c\in \mathbb{R}$ is arbitrarily fixed. Let $X_1=\grad K /|\grad K|$ and $X_2\in C(TM)$ be two vector fields on $M$ such that $\left\{X_1(p),X_2(p)\right\}$ is a positively oriented basis at any point $p\in M$. Then, the following conditions are equivalent:
\begin{itemize}
  \item [(i)] the level curves of $K$ are circles in $M$ with constant curvature
  $$
  \kappa=\frac{3|\grad K|}{8(c-K)}=\frac{3X_1K}{8(c-K)};
  $$
  \item [(ii)]
  $$
  X_2\left(X_1K\right)=0 \quad \text{and} \quad \nabla_{X_2}X_2=\frac{-3X_1K}{8(c-K)}X_1;
  $$
  \item [(iii)]
  $$
  \nabla_{X_1}X_1=\nabla_{X_1}X_2=0,\quad \nabla_{X_2}X_2=-\frac{3X_1K}{8(c-K)}X_1,\quad \nabla_{X_2}X_1=\frac{3X_1K}{8(c-K)}X_2.
  $$
  \item [(iv)] the metric $g$ can be locally written, as $g=e^{2\sigma}\left(du^2+dv^2\right)$, where $(u,v)$ are positively oriented local coordinates, and $\sigma=\sigma(u)$ satisfies the equation
      $$
      \sigma^{\prime\prime}=e^{-2\sigma/3}-ce^{2\sigma}
      $$
      and the condition $\sigma^\prime>0$; moreover, the solutions of the above equation, $u=u(\sigma)$, are
      $$
      u=\int_{\sigma_0}^{\sigma}\frac{d\tau}{\sqrt{-3e^{-2\tau/3}-ce^{2\tau}+a}}+u_0,
      $$
      where $\sigma$ is in some open interval $I$, $\sigma_0\in I$ and $a,u_0\in \mathbb{R}$ are constants;

\end{itemize}
\end{theorem}

\section{The intrinsic approach}\label{sec-Intrinsic}

From Theorem \ref{thm:char} and Theorem \ref{thm:carac}, we have the following local intrinsic characterization of biconservative surfaces in three-dimensional space forms, and in particular in the hyperbolic space $\mathbb{H}^3$ ($c=-1$). That is, if we consider an abstract surface $\left(M^2,g\right)$ with $-1-K(p)>0$ and $(\grad K)(p)\neq 0$ at any point $p\in M$, then locally it admits a (unique) biconservative immersion in $\mathbb{H}^3$ with a nowhere vanishing gradient of the mean curvature, if and only if, locally, the metric $g$ can be written as $g(u,v)=e^{2\sigma(u)}\left(du^2+dv^2\right)$, where $\sigma'(u)\neq 0$, for any $u$, and $u=u(\sigma)$ is given by
$$
u(\sigma)=\int_{\sigma_0}^{\sigma}\frac{d\tau}{\sqrt{-3e^{-2\tau/3}+e^{2\tau}+a}}+u_0, \qquad a,u_0\in\mathbb{R}.
$$
With the new coordinates $(\sigma,v)$ the metric $g$ can be locally written as
$$
g(\sigma,v)=e^{2\sigma}\left(\frac{1}{-3e^{-2\sigma/3}+e^{2\sigma}+a}d\sigma^2+dv^2\right),
$$
and we have a one parameter family of such metrics. In order to find a more convenient expression for the metric $g$, we will change the coordinates twice. First we take $(\sigma,v)=\left(\log\left(3^{3/4}/\xi\right),v\right)$, $\xi>0$. Denoting $C_{-1}=a/\sqrt{3}\in\mathbb{R}$, one obtains
$$
g_{C_{-1}}(\xi,v)=\frac{1}{\xi^2}\left(\frac{3}{-\xi^{8/3}+C_{-1}\xi^2+3}d\xi^2+3\sqrt{3}dv^2\right).
$$
The second change of coordinates is given by $(\xi,v)=\left(\xi,\theta/3^{3/4}\right)$, and we have
$$
g_{C_{-1}}(\xi,\theta)=\frac{1}{\xi^2}\left(\frac{3}{-\xi^{8/3}+C_{-1}\xi^2+3}d\xi^2+d\theta^2\right),
$$
where $C_{-1}$, $\theta\in\mathbb{R}$ and $\xi$ is positive and belongs to an open interval such that $-\xi^{8/3}+C_{-1}\xi^2+3>0$. In order to determine the largest interval for $\xi$, we define the function $T:(0,\infty)\to\mathbb{R}$ such that
$$
T(\xi)=-\xi^{8/3}+C_{-1}\xi^2+3
$$
and we try to find where $T$ is positive. By some standard computations, we come to the following three cases:
\begin{itemize}
  \item if $C_{-1}>0$, we get that there exists a unique point $\xi_{01}$ depending on $C_{-1}$, $\xi_{01}>\left(3C_{-1}/4\right)^{3/2}$ such that the function $T$ vanishes at this point, $T(\xi)>0$ for any $\xi\in\left(0,\xi_{01}\right)$ and $T(\xi)<0$ for any $\xi\in\left(\xi_{01},\infty\right)$.
  \item if $C_{-1}<0$, one obtains that there exists a point $\xi_{01}>0$ (we keep the same notation for the vanishing point) such that the function $T$ vanishes at this point and $T$ is positive on the interval $\left(0,\xi_{01}\right)$.
  \item if $C_{-1}=0$, one gets that there exists a point $\xi_{01}=3^{3/8}$ such that the function $T$ vanishes at this point and $T$ is positive on $\left(0,\xi_{01}\right)$.
\end{itemize}
Therefore, in all cases, we have,
$$
g_{C_{-1}}(\xi,\theta)=\frac{1}{\xi^2}\left(\frac{3}{-\xi^{8/3}+C_{-1}\xi^2+3}d\xi^2+d\theta^2\right), \qquad (\xi,\theta)\in\left(0,\xi_{01}\right)\times \mathbb{R},
$$
and we get the following result.

\begin{theorem}[\cite{NO17,NPhD17}]
Let $\left(M^2,g(u,v)=e^{2\sigma(u)}\left(du^2+dv^2\right)\right)$ be an abstract surface, where $u=u(\sigma)$ is given by
$$
u(\sigma)=\int_{\sigma_0}^{\sigma}\frac{d\tau}{\sqrt{-3e^{-2\tau/3}+e^{2\tau}+a}}+u_0, \qquad \sigma\in I,
$$
where $a$ and $u_0$ are real constants and $I$ is an open interval. Then $\left(M^2,g\right)$ is isometric to
$$
\left(D_{C_{-1}},g_{C_{-1}}\right)=\left(\left(0,\xi_{01}\right)\times\mathbb{R}, g_{C_{-1}}(\xi,\theta)=\frac{1}{\xi^2}\left(\frac{3}{-\xi^{8/3}+C_{-1}\xi^2+3}d\xi^2+d\theta^2\right)
\right),
$$
where $C_{-1}$ is a real constant and $\xi_{01}$ is the positive vanishing point of $-\xi^{8/3}+C_{-1}\xi^2+3$.
\end{theorem}

\begin{remark}
We call the surface $\left(D_{C_{-1}},g_{C_{-1}}\right)$ an abstract standard biconservative surface, and, in fact, we have a one-parameter family of abstract standard biconservative surfaces indexed by $C_{-1}$.
\end{remark}

\begin{remark}
We note that
$$
\lim_{\xi\searrow 0}\frac{3}{\xi^2\left(-\xi^{8/3}+C_{-1}\xi^2+3\right)}= \lim_{\xi\nearrow\xi_{01}}\frac{3}{\xi^2\left(-\xi^{8/3}+C_{-1}\xi^2+3\right)}=\infty,
$$
and therefore, the metric $g_{C_{-1}}$ blows up at the boundary given by $\xi=0$ and $\xi=\xi_{01}$.
\end{remark}

The surface $\left(D_{C_{-1}},g_{C_{-1}}\right)$ is not complete since the geodesic $\theta=\theta_0$ cannot be defined on the whole $\mathbb{R}$ but only on a half line, and by standard computations it can be proved that its Gaussian curvature is given by
\begin{equation}\label{eq:K-DC}
K_{C_{-1}}(\xi,\theta)=K(\xi)=-\frac{\xi^{8/3}}{9}-1
\end{equation}
and
$$
\qquad K'(\xi)=-\frac{8}{27}\xi^{5/3}<0.
$$
Therefore
\begin{equation}\label{eq:gradK-DC}
\grad K=\frac{\xi^2\left(-\xi^{8/3}+C_{-1}\xi^2+3\right)}{3} K'(\xi)\frac{\partial}{\partial \xi}
\end{equation}
does not vanish at any point of $D_{C_{-1}}$ and
$$
\lim_{\xi\nearrow\xi_{01}}(\grad K)(\xi,\theta)=0, \qquad \theta\in\mathbb{R}.
$$

As the metric $g_{C_{-1}}$ is not complete, in order to obtain a complete one, denoted by $\tilde{g}_{C_{-1}}$, or simply $\tilde{g}$, we will change the coordinates again and then glue, in a simple way, two (isometric) metrics $g_{C_{-1}}$. So, if we consider the change of coordinates given by $(\xi,\theta)=\left(\xi(\rho),\theta\right)$, one obtains
$$
g_{C_{-1}}(\rho,\theta)=\frac{1}{\xi^2(\rho)}d\theta^2+d\rho^2,
$$
where $\xi=\xi(\rho)$ is the inverse function of $\rho$,
$$
\rho(\xi)=-\int_{\xi_{00}}^{\xi}\sqrt{\frac{3}{\tau^2\left(-\tau^{8/3}+C_{-1}\tau^2+3\right)}}\ d\tau,
$$
$\xi_{00}$ being an arbitrarily fixed constant in $\left(0,\xi_{01}\right)$.

We are allowed to consider the above change as $\rho$ is a strictly decreasing function. Moreover, we have the following lemma.

\begin{lemma}\label{lemma1}
The function $\rho:\left(0,\xi_{01}\right)\to\mathbb{R}$ defined by
$$
\rho(\xi)=-\int_{\xi_{00}}^{\xi}\sqrt{\frac{3}{\tau^2\left(-\tau^{8/3}+C_{-1}\tau^2+3\right)}}\ d\tau
$$
satisfies
$$
\lim_{\xi\searrow 0} \rho(\xi)= \infty, \qquad \lim_{\xi\nearrow \xi_{01}} \rho(\xi)=\rho_1,
$$
where $\rho_1$ is a negative real constant.
\end{lemma}

\begin{proof}
In order to compute the first limit, we change the variable $\tau=1/\tilde{\tau}$ in the integral
$$
I=-\int_{\xi_{00}}^{\xi}\sqrt{\frac{3}{\tau^2\left(-\tau^{8/3}+C_{-1}\tau^2+3\right)}}\ d\tau,
$$
and obtain
$$
I=\int_{1/\xi_{00}}^{1/\xi}F(\tilde{\tau})\ d\tilde{\tau},
$$
where
$$
F(\tilde{\tau})=\frac{\sqrt{3} \tilde{\tau}^{1/3}}{\sqrt{3\tilde{\tau}^{8/3}+C_{-1}\tilde{\tau}^{2/3}-1}}, \qquad \tilde{\tau}\in \left[\frac{1}{\xi_{00}},\infty\right),
$$
for $\xi<\xi_{00}$.

Since
$$
\lim_{\tilde{\tau}\to\infty}\tilde{\tau}F(\tilde{\tau})=1\in(0,\infty],
$$
it follows that
$$
\lim_{\xi\searrow 0} \rho(\xi)= \infty.
$$
In order to compute the second limit, we first note that $\rho(\xi)$ is negative for any $\xi\in \left(\xi_{00},\xi_{01}\right)$ and
$$
\rho(\xi)>-\frac{\sqrt{3}}{\xi_{00}}\int_{\xi_{00}}^{\xi}\frac{1}{\sqrt{T(
\tau)}}\ d\tau, \qquad \xi\in \left(\xi_{00},\xi_{01}\right),
$$
where $T(\xi)=-\xi^{8/3}+C_{-1}\xi^2+3$.

We have that that $\lim_{\xi\nearrow \xi_{01}} \rho(\xi)$ is finite if and only if
$$
\lim_{\xi\nearrow\xi_{01}}\int_{\xi_{00}}^{\xi}\frac{1}{\sqrt{T(\tau)}}\ d\tau <\infty.
$$
To prove this, we rewrite the function $T$ as
\begin{align*}
  T(\xi)= & T\left(\xi_{01}\right)+\left(\xi-\xi_{01}\right)\left(T'\left(\xi_{01}\right)+\alpha_0(\xi)\right) \\
  = & \left(\xi-\xi_{01}\right)T_1(\xi),
\end{align*}
where $\alpha_0$ is a continuous function such that $\lim_{\xi\nearrow \xi_{01}}\alpha_0(\xi)=0$ and
$$
T_1(\xi)= -\frac{8}{3}\xi^{5/3}_{01}+2C_{-1}\xi_{01}+\alpha_0(\xi), \qquad \xi\in \left[\xi_{00},\xi_{01}\right].
$$
Then, we have
$$
\lim_{\xi\nearrow\xi_{01}}\sqrt{\xi_{01}-\xi}\frac{1}{\sqrt{T(\xi)}}=  \lim_{\xi\nearrow\xi_{01}}\frac{1}{\sqrt{-T_1(\xi)}}\in[0,\infty),
$$
and we come to the conclusion that
$$
\lim_{\xi\nearrow \xi_{01}} \rho(\xi)=\rho_1\in\mathbb{R}^\ast_{-}.
$$
\end{proof}

Denoting $\tilde{h}(\rho)=1/\xi(\rho)$, the metric $g_{C_{-1}}$ can be rewritten as
$$
g_{C_{-1}}(\rho,\theta)=\tilde{h}^2(\rho) d\theta^2+d\rho^2, \qquad  (\rho,\theta)\in \left(\rho_1,\infty\right)\times\mathbb{R}.
$$
We can obtain a simpler form of the domain, i.e., $\left(0,\infty\right)\times\mathbb{R}$, considering a new change of coordinates given by $(\rho,\theta)=\left(\rho(\omega)=\omega+\rho_1,\theta\right)$. Therefore, we have
$$
g_{C_{-1}}(\omega,\theta)=h^2(\omega) d\theta^2+d\omega^2, \qquad  (\omega,\theta)\in \left(0,\infty\right)\times\mathbb{R},
$$
where $h(\omega)=\tilde{h}(\rho(\omega))$.
\begin{remark}
We note that
$$
\lim_{\omega\searrow 0}h^2(\omega)= \frac{1}{\xi_{01}^2}\in\mathbb{R}^\ast_{+},
$$
and thus, the metric $g_{C_{-1}}$ can be smoothly extended to the boundary $\omega=0$.
\end{remark}

As the limit of the function $h$ when $\omega$ approaches $0$ is $1/\xi_{01}\neq 0$, it is easy to see that $\left(\left[0,\infty\right)\times\mathbb{R}, g_{C_{-1}}\right)$ can be viewed as a surface with boundary.

In order to obtain a complete surface, we extend the surface $\left(\left(0,\infty\right)\times\mathbb{R}, g_{C_{-1}}\right)$ by ``symmetry'' with respect to its boundary and get the following result.

\begin{theorem}\label{th-metricComplete}
The surface $\left(\mathbb{R}^2,\tilde{g}_{C_{-1}}(\omega,\theta)=\Gamma^2(\omega)d\theta^2+d\omega^2\right)$ is complete,
where the function $\Gamma:\mathbb{R}\to\mathbb{R}$ is given by
\begin{equation}\label{ec:gammma}
\Gamma(\omega)=\left\{
\begin{array}{ll}
  h(\omega), & \omega>0 \\\\
  \frac{1}{\xi_{01}}, & \omega=0 \\\\
  h(-\omega), & \omega<0
\end{array}
\right..
\end{equation}
\end{theorem}

\begin{proof}
By some standard computations it is easy to verify that the function $\Gamma$ is at least of class $C^3$. In order to prove that the metric
$\tilde{g}_{C_{-1}}$ is complete, we first note that $\Gamma(\omega)\geq 1/\xi_{01}$, for any $\omega\in\mathbb{R}$, and then consider the metric
$$
\tilde{g}_0(\omega,\theta)=m_0\left(d\theta^2+d\omega^2\right), \qquad (\omega,\theta)\in\mathbb{R}^2,
$$
where $m_0$ is the minimum between $1/\xi_{01}^2$ and $1$. As the metric $\tilde{g}_0$ is complete and $\tilde{g}_{C_{-1}}-\tilde{g}_0$ is non-negative at any point of the surface, it follows that $\tilde{g}_{C_{-1}}$ is complete (see \cite{G73}).
\end{proof}

\begin{remark}
Since $\left(\grad \tilde{K}\right)(0,\theta)=0$, for any $\theta\in\mathbb{R}$, where $\tilde{K}$ is the Gaussian curvature of $\left(\mathbb{R}^2,\tilde{g}_{C_{-1}}\right)$, it follows that $\nabla_{\frac{\partial}{\partial \theta}} \frac{\partial}{\partial \theta}=0$ along the boundary of $\left((0,\infty)\times\mathbb{R}, g_{C_{-1}}\right)$ and therefore this boundary becomes a geodesic in $\left(\mathbb{R}^2,\tilde{g}_{C_{-1}}\right)$.
\end{remark}

\begin{remark}
A similar construction is also possible when $c=0$ or $c=1$. In the first case, in order to obtain an abstract complete biconservative surface we glue two abstract standard biconservative surfaces, and in the second case, the gluing process must be performed infinitely many times. In fact, for $c=0$ we will reobtain Theorem 4.1.  from \cite{N16} (where the complete surface was obtained by working with isothermal coordinates), and for $c=1$ we will reobtain Proposition 4.17.  from \cite{N16} (where the main idea was that the abstract standard biconservative surface is isometric to a certain surface of revolution in the $3$-dimensional Euclidean space $\mathbb{R}^3$).
\end{remark}

We also note that, since the Gaussian curvature of the complete surface $\left(\mathbb{R}^2,\tilde{g}_{C_{-1}}\right)$ satisfies $(\grad \tilde{K}_{C_{-1}})(0,\theta)=0$, for any $\theta\in\mathbb{R}$, the existence of a (non-$CMC$) biconservative immersion from $\left(\mathbb{R}^2,\tilde{g}_{C_{-1}}\right)$ in $\mathbb{H}^3$ is not guaranteed. So, our aim is to construct such an immersion.

For the sake of simplicity, we will omit writing the index $C_{-1}$ in the following construction. Let us denote by
$$
^1g(\omega,\theta)=h^2(\omega)d\theta^2+d\omega^2, \qquad (\omega,\theta)\in (0,\infty)\times\mathbb{R}
$$
and
$$
^2g(\omega,\theta)=h^2(-\omega)d\theta^2+d\omega^2, \qquad (\omega,\theta)\in (-\infty,0)\times\mathbb{R}.
$$
It is easy to see that the Gaussian curvatures of the above two surfaces are given by
$$
^1K(\omega)=-\frac{h''(\omega)}{h(\omega)}, \ \omega\in (0,\infty), \qquad ^2K(\omega)=\ ^1K(-\omega)=-\frac{h''(-\omega)}{h(-\omega)},\ \omega\in (-\infty,0)
$$
and their derivatives are equal to
$$
^1K'(\omega)=\frac{-h'''(\omega)h(\omega)+h''(\omega)h'(\omega)}{h^2(\omega)}, \qquad \omega\in (0,\infty),
$$
$$
^2K'(\omega)=-\ ^1K'(-\omega)=\frac{h'''(-\omega)h(-\omega)-h''(-\omega)h'(-\omega)}{h^2(-\omega)}, \qquad \omega\in (-\infty,0),
$$
respectively.

Let
$$
^1X_1=\frac{\grad \ ^1K}{\left|\grad \ ^1K\right|}, \qquad  ^2X_1=\frac{\grad \ ^2K}{\left|\grad \ ^2K\right|},
$$
be two vector fields defined on $(0,\infty)\times\mathbb{R}$, respectively on $(-\infty,0)\times\mathbb{R}$. Clearly, as $^1K'>0$, one obtains:
$$
^1X_1=\frac{\partial}{\partial \omega} \text{ and } ^2X_1=-\frac{\partial}{\partial \omega}
$$
on $(0,\infty)\times\mathbb{R}$, respectively on $(-\infty,0)\times\mathbb{R}$.

Now, let us define, on $\mathbb{R}^2$,
\begin{equation}\label{eq:vect_field_X1-1}
X_1(\omega,\theta)=\left\{
\begin{array}{cc}
  ^1X_1(\omega,\theta), &  (\omega,\theta)\in (0,\infty)\times\mathbb{R}\\\\
  \frac{\partial}{\partial \omega}, &  (\omega,\theta)\in\{0\}\times\mathbb{R}\\\\
  -^2X_1(\omega,\theta), & (\omega,\theta)\in (-\infty,0)\times\mathbb{R}
\end{array}
\right..
\end{equation}
Clearly, the vector field $X_1$ is given by $X_1=\frac{\partial}{\partial \omega}$ on $\mathbb{R}^2$.

Now, the vector field $X_1$ determines uniquely the global vector field $X_2$ by asking $\left\{X_1(\omega,\theta),X_2(\omega,\theta)\right\}$ to be a positive orthonormal frame field in $\left(\mathbb{R}^2,\tilde{g}\right)$, for any $(\omega,\theta)\in \mathbb{R}^2$. Obviously, $X_2(\omega,\theta)=\  ^1X_2(\omega,\theta)$, for any $(\omega,\theta)\in (0,\infty)\times\mathbb{R}$, and $X_2(\omega,\theta)=-\  ^2X_2(\omega,\theta)$, for any $(\omega,\theta)\in (-\infty,0)\times\mathbb{R}$, that is $X_2(\omega,\theta)=\frac{1}{\Gamma(\omega)}\frac{\partial}{\partial \theta}$ on $\mathbb{R}^2$, where $\Gamma$ is defined in \eqref{ec:gammma}.

Further, we have the following properties of $X_1$ and $X_2$.

\begin{proposition}\label{prop-1}
Let $\left(\mathbb{R}^2,\tilde{g}\right)$ the above complete surface. Then, the Gaussian curvature $\tilde{K}$ of $\left(\mathbb{R}^2,\tilde{g}\right)$ satisfies $-1- \tilde{K}>0$ at any point, and the vector fields $X_1$ and $X_2$ defined above, satisfy on $\mathbb{R}^2$
$$
\nabla_{X_1}{X_1}=\nabla_{X_1}{X_2}=0,\quad \nabla_{X_2}{X_2}=-\frac{3X_1\tilde{K}}{8\left(-1-\tilde{K}\right)}X_1, \quad \nabla_{X_2}{X_1}=\frac{3X_1\tilde{K}}{8\left(-1-\tilde{K}\right)}X_2.
$$
\end{proposition}

\begin{proof}
We recall that the surface $\left((0,\infty)\times\mathbb{R},\ ^1g\right)$ has the following properties: $-1- \ ^1K>0$, $\grad \ ^1K\neq 0$ at any point, and on $(0,\infty)\times\mathbb{R}$ one has
{\small $$
\nabla_{^1X_1}{^1X_1}=\nabla_{^1X_1}{^1X_2}=0,\ \nabla_{^1X_2}{^1X_2}=-\frac{3\ ^1X_1\ ^1K}{8\left(-1-\ ^1K\right)}\ ^1X_1,
  \ \nabla_{^1X_2}{^1X_1}=\frac{3\ ^1X_1\ ^1K}{8\left(-1-\ ^1K\right)}\ ^1X_2.
$$}
It is easy to see that $\left((-\infty,0)\times\mathbb{R},\ ^2g\right)$ has the same properties as $\left((0,\infty)\times\mathbb{R},\ ^1g\right)$.

Therefore, on $(-\infty,0)\times\mathbb{R}$, we have
{\small
$$
\nabla_{^2X_1}{^2X_1}=\nabla_{^2X_1}{^2X_2}=0,\ \nabla_{^2X_2}{^2X_2}=-\frac{3\ ^2X_1\ ^2K}{8\left(-1-\ ^2K\right)}\ ^2X_1,
  \ \nabla_{^2X_2}{^2X_1}=\frac{3\ ^2X_1\ ^2K}{8\left(-1-\ ^2K\right)}\ ^2X_2.
$$
}
From the definition of $X_1$ and $X_2$, we note that, on $\mathbb{R}^\ast\times\mathbb{R}$, we have
\begin{equation}\label{eq: LV-connec}
\nabla_{X_1}{X_1}=\nabla_{X_1}{X_2}=0,\quad \nabla_{X_2}{X_2}=-\frac{3X_1\tilde{K}}{8\left(-1-\tilde{K}\right)}X_1, \quad \nabla_{X_2}{X_1}=\frac{3X_1\tilde{K}}{8\left(-1-\tilde{K}\right)}X_2,
\end{equation}
where
\begin{equation*}
\tilde{K}(\omega,\theta)=\tilde{K}(\omega)=\left\{
\begin{array}{ll}
  ^1K(\omega), & (\omega,\theta)\in(0,\infty)\times\mathbb{R} \\\\
  ^2K(\omega), &(\omega, \theta)\in (-\infty,0)\times\mathbb{R}
\end{array}
\right.,
\end{equation*}
$\tilde{K}$ being the Gaussian curvature of $\left(\mathbb{R}^2,\tilde{g}\right)$. Then,
$$
\tilde{K}(0)=\lim_{\omega\searrow 0} \ ^1K(\omega)=\lim_{\omega\nearrow 0}\  ^2K(\omega)=-\frac{1}{9}\xi_{01}^{8/3}-1.
$$
Moreover, $-1-\tilde{K}>0$ on $\mathbb{R}^2$ and all the objects defined in \eqref{eq: LV-connec} are, in fact, defined on $\mathbb{R}^2$, and they are at least continuous. So, passing to the limit when $\omega$ approaches $0$ we obtain that \eqref{eq: LV-connec} holds on whole $\mathbb{R}^2$.
\end{proof}

Now, we can state the following existence and uniqueness result.

\begin{theorem}\label{th:EUH3}
Let $\left(\mathbb{R}^2,\tilde{g}\right)$ the above complete surface. Then, there exists a unique biconservative immersion $\Phi:\left(\mathbb{R}^2,\tilde{g}\right)\to\mathbb{H}^3$. Moreover, $\grad f\neq 0$ at any point of $\mathbb{R}^\ast\times\mathbb{R}$, where $f$ is the mean curvature function of the immersion $\Phi$.
\end{theorem}

\begin{proof}
First, we note that, from Proposition \ref{prop-1}, we have that the vector fields $X_1$ and $X_2$ on $\mathbb{R}^2$, previously defined, satisfy \eqref{eq: LV-connec} on $\mathbb{R}^2$.

In order to prove the existence of a biconservative immersion $\Phi:\left(\mathbb{R}^2,\tilde{g}\right)\to\mathbb{H}^3$, let us consider the operator $A:C\left(T\mathbb{R}^2\right)\to C\left(T\mathbb{R}^2\right)$ defined by
$$
A\left(X_1\right)=-\frac{\sqrt{-1-\tilde{K}}}{\sqrt{3}}X_1, \qquad A\left(X_2\right)=\sqrt{3\left(-1-\tilde{K}\right)}X_2.
$$
We will prove that $A$ satisfies the Gauss and the Coddazi equations. Since the matrix of $A$ with respect to $\left\{X_1,X_2\right\}$ is
\begin{equation*}
A=
\left(
\begin{array}{cc}
          -\frac{\sqrt{-1-\tilde{K}}}{\sqrt{3}} & 0 \\
          0 & \sqrt{3\left(-1-\tilde{K}\right)}
        \end{array}
\right),
\end{equation*}
it is easy to see that $\det A=1+\tilde{K}$, i.e., the Gauss equation is satisfied, and
$$
f=\trace A=\frac{2}{\sqrt{3}}\sqrt{-1-\tilde{K}}.
$$
By some direct computations, also using \eqref{eq: LV-connec}, one obtains that
$$
\left(\nabla_{X_1}A\right)\left(X_2\right)=\left(\nabla_{X_2}A\right)\left(X_1\right),
$$
i.e., the Codazzi equation.

Therefore, from the fundamental theorem of surfaces in $\mathbb{H}^3$, it follows that there exists an unique isometric immersion $\Phi:\left(\mathbb{R}^2,\tilde{g}\right)\to\mathbb{H}^3$ such that $A$ is its shape operator. Moreover, the operator $A$ satisfies
$$
A(\grad f)=-\frac{f}{2}\grad f,
$$
which shows that $\Phi$ is biconservative.

Further, we will prove the uniqueness of biconservative immersions from $\left(\mathbb{R}^2,\tilde{g}\right)$ in $\mathbb{H}^3$. Let $\Phi_1$ and  $\Phi_2$ two biconservative immersions from $\left(\mathbb{R}^2,\tilde{g}\right)$ in $\mathbb{H}^3$. Obviously, ${\Phi_1}_{\left|(0,\infty)\times\mathbb{R}\right.}$ and ${\Phi_1}_{\left|(-\infty,0)\times\mathbb{R}\right.}$ are biconservative, and therefore using Theorem \ref{thm:reformulate}, it follows that $\grad f\neq 0$ on $\left((0,\infty)\times\mathbb{R},\tilde{g}\right)$ and on $\left((-\infty,0)\times\mathbb{R},\tilde{g}\right)$ and these restrictions are unique (up to isometries of $\mathbb{H}^3$).

It follows that there exist two isometries $^1F$ and $^2F$ of $\mathbb{H}^3$, which preserve the orientation, such that
$$
{\Phi_2}_{\left|(0,\infty)\times\mathbb{R}\right.}= \ ^1F\circ {\Phi_1}_{\left|(0,\infty)\times\mathbb{R}\right.}
$$
and
$$
{\Phi_2}_{\left|(-\infty,0)\times\mathbb{R}\right.}= \ ^2F\circ {\Phi_1}_{\left|(-\infty,0)\times\mathbb{R}\right.}
$$
By continuity, one obtains that $^1F$ and $^2F$ coincide along the curve $\theta\to\Phi_1(0,\theta)$, for any $\theta\in\mathbb{R}$. By a straightforward computation, we have

\begin{align}\label{eq:F1F2-1}
^1F_{\ast,\Phi_1\left(0,\theta_0\right)}\left(\Phi_{1\ast,\left(0,\theta_0\right)}\left(\left(\frac{\partial}{\partial \theta}\right)_{\left(0,\theta_0\right)}\right)\right)= & \ ^2F_{\ast,\Phi_1\left(0,\theta_0\right)}\left(\Phi_{1\ast,\left(0,\theta_0\right)}\left(\left(\frac{\partial}{\partial \theta}\right)_{\left(0,\theta_0\right)}\right)\right)\nonumber\\
= & \Phi_{2\ast,\left(0,\theta_0\right)}\left(\left(\frac{\partial}{\partial \theta}\right)_{\left(0,\theta_0\right)}\right),
\end{align}
where $\theta_0\in\mathbb{R}$ is a given number.

We note that
\begin{equation*}
  \left\{
  \begin{array}{ll}
    \Phi_2 \left(\omega,\theta_0\right)= \ ^1F\left(\Phi_1 \left(\omega,\theta_0\right)\right), & \omega>0 \\\\
    \Phi_2 \left(\omega,\theta_0\right)= \ ^2F\left(\Phi_1 \left(\omega,\theta_0\right)\right), & \omega<0
  \end{array}
  \right.,
\end{equation*}
and, again by continuity, we have
\begin{equation*}
  \left\{
  \begin{array}{ll}
    \Phi_2 \left(\omega,\theta_0\right)= \ ^1F\left(\Phi_1 \left(\omega,\theta_0\right)\right), & \omega\geq 0 \\\\
    \Phi_2 \left(\omega,\theta_0\right)= \ ^2F\left(\Phi_1 \left(\omega,\theta_0\right)\right), & \omega\leq 0
  \end{array}
  \right..
\end{equation*}
It follows that
\begin{align}\label{eq:F1F2-2}
^1F_{\ast,\Phi_1\left(0,\theta_0\right)}\left(\Phi_{1\ast,\left(0,\theta_0\right)}\left(\left(\frac{\partial}{\partial \omega}\right)_{\left(0,\theta_0\right)}\right)\right)= & \ ^2F_{\ast,\Phi_1\left(0,\theta_0\right)}\left(\Phi_{1\ast,\left(0,\theta_0\right)}\left(\left(\frac{\partial}{\partial \omega}\right)_{\left(0,\theta_0\right)}\right)\right)\nonumber\\
= & \Phi_{2\ast,\left(0,\theta_0\right)}\left(\left(\frac{\partial}{\partial \omega}\right)_{\left(0,\theta_0\right)}\right),
\end{align}
From equations \eqref{eq:F1F2-1} and \eqref{eq:F1F2-2}, it is clear that $^1F_{\ast,\Phi_1\left(0,\theta_0\right)}$ and $^2F_{\ast,\Phi_1\left(0,\theta_0\right)}$ coincide on $\Phi_{1\ast,\left(0,\theta_0\right)}\left(\mathbb{R}^2\right)$. As $^1F$ and $^2F$ preserve the orientation of $\mathbb{H}^3$, it results that $^1F_{\ast,\Phi_1\left(0,\theta_0\right)}$ and $^2F_{\ast,\Phi_1\left(0,\theta_0\right)}$ coincide on $\left[\Phi_{1\ast,\left(0,\theta_0\right)}\left(\mathbb{R}^2\right)\right]^\perp$. So,
$^1F_{\ast,\Phi_1\left(0,\theta_0\right)}=\ ^2F_{\ast,\Phi_1\left(0,\theta_0\right)}$ and therefore, since $^1F$ and $^2F$ also agrees at least at one point, we come to the conclusion.
\end{proof}

\begin{remark}
Similar results can be proved when $c=0$ or $c=1$. In fact, the existence part of these theorems were, essentially, already obtained in Theorem 4.1 and in Theorem 4.18 from \cite{N16}, respectively, by a direct construction.
\end{remark}

\vspace{0.5cm}
We end this section with the following conjectures.

\vspace{0.25cm}
\textbf{Conjecture 1.} Let $\varphi:M^2\to\mathbb{H}^3$ be a simply connected, complete, non-$CMC$ biconservative surface. Then, up to isometries of the domain and codomain, $M$ and $\varphi$ are those given in Theorem \ref{th:EUH3}.

\vspace{0.25cm}
In fact, Conjecture 1 applies for any $3$-dimensional space form $N^3(c)$.

\vspace{0.25cm}
\textbf{Conjecture 2.} Any compact biconservative surface in $N^3(c)$ is $CMC$.

\vspace{0.25cm}

If the Conjecture 1 is true for $N^3(c)$, the Conjecture 2 is also true for $c=0$ and for $c=-1$ (as we will see in the next section) by considering the universal cover and taking into account that the corresponding biconservative immersion is not double periodic. The only interesting case is $c=1$, where the domain quotients to a non-flat torus but it is not clear if the immersion is double periodic.

We note that the above conjectures were positively answered when the target manifold is the three-dimensional Euclidean space $\mathbb{R}^3$ and $\varphi$ is an embedding (see \cite{NPhD17, NO19}). Using the same technique, we expect the same kind of results to hold, i.e., if $\varphi:\left(M^2,g\right)\to\mathbb{H}^3$ is an embedding and $M$ is complete, then $\varphi(M)$ is one of the complete biconservative surfaces constructed extrinsically in the next section.

\section{The extrinsic approach}\label{sec-Extrinsic}

The aim of this section is to construct complete, non-$CMC$ biconservative surfaces in $\mathbb{H}^3$. The idea is to glue the images of two standard biconservative surfaces (the abstract domain of the two parametrizations remains the same), reobtaining basically, in an extrinsic way, the existence part of Theorem \ref{th:EUH3}.

For the sake of completeness, we firs present some local extrinsic properties of biconservative surfaces.

\begin{theorem}[\cite{CMOP14}]
Let $M^2$ be a biconservative surface in $N^3(c)$ with a nowhere vanishing gradient of the mean curvature function $f$. Then, one has $f>0$ and
\begin{equation}\label{eq:bicons1}
f\Delta f+|\grad f|^2+\frac{4}{3}cf^2-f^4=0,
\end{equation}
where $\Delta$ is the Laplace-Beltrami operator on $M$.
\end{theorem}

In the same paper \cite{CMOP14}, it was proved that, from equation \eqref{eq:bicons1}, it follows that there exists a positively oriented local chart $(U;u,v)$ such that $f=f(u,v)=f(u)$ satisfies the following ODE:
\begin{equation}\label{eq:bicons2}
ff''=\frac{7}{4}f'^2+\frac{4}{3}cf^2-f^4.
\end{equation}
Then, denoting $\kappa(u)=f(u)/2$, from equation \eqref{eq:bicons2} one obtains that $\kappa$ satisfies
$$
\kappa\kappa''=\frac{7}{4}\kappa'^2+\frac{4}{3}c\kappa^2-4\kappa^4.
$$
and
\begin{equation}\label{eq:firstIntegral1}
\kappa'^2=-\frac{16}{9}c\kappa^2-16\kappa^4+\tilde{C}_{-1}\kappa^{7/2},
\end{equation}
where $\tilde{C}_{-1}$ is a real constant.

Next, we work in the $3$-dimensional hyperbolic space $\mathbb{H}^3$, i.e., $c=-1$. As there exist several models for the hyperbolic space, in this paper we will consider, in each particular situation, the most appropriate model in order to obtain a complete biconservative surface.

We note that a local extrinsic characterization of biconservative surfaces in $\mathbb{H}^3$ was given in \cite{CMOP14} where the authors considered the hyperboloid model of $\mathbb{H}^3$. Let us recall that the Minkowski space $\mathbb{R}^{4}_{1}$ is given by $\mathbb{R}^{4}_{1}=\left(\mathbb{R}^4,\langle\cdot,\cdot\rangle\right)$, where $\langle\cdot,\cdot\rangle$ is the bilinear form
$$
\langle x,y\rangle=\sum_{i=1}^3 x^iy^i-x^4y^4, \qquad x=\left(x^1,x^2,x^3,x^4\right),\quad y=\left(y^1,y^2,y^3,y^4\right).
$$
The hyperboloid model is
$$
\mathbb{H}^3=\left\{x\in \mathbb{R}^{4}_{1} \quad : \quad \langle x,x\rangle=-1 \text{\  and \ } x^4>0\right\},
$$
that is the the upper part of the hyperboloid of two sheets.

It is well known that the Levi-Civita connections $\overline{\nabla}$ of $\mathbb{R}^{4}_{1}$, and $\nabla^{\mathbb{H}^3}$ of $\mathbb{H}^3$, are related by
\begin{equation}\label{eq1}
\overline{\nabla}_X Y=\nabla^{\mathbb{H}^3}_X Y+\langle X,Y\rangle x,
\end{equation}
for any $X,Y\in C(T\mathbb{H}^3)$ and $x\in \mathbb{H}^3$.

Further, let $\varphi:M^2\to\mathbb{H}^3$ be a connected, oriented biconservative surface where on $\mathbb{H}^3$ we considered the Riemannian metric induced by the pseudo-Riemannian metric on $\mathbb{R}^{4}_{1}$. If we assume that $\grad f$ is nowhere vanishing and consider the global orthonormal frame field in $\mathbb{H}^3$ along $M$, $\left\{X_1=\grad f/ |\grad f|, X_2, \eta \right\}$,  then the Levi-Civita connection $\nabla$ on M, is given by
\begin{equation}\label{eq3}
\begin{array}{c}
\nabla_{X_1}X_1=\nabla_{X_1}X_2=0, \qquad
\nabla_{X_2}X_1=-\frac{3X_1f}{4f}X_2, \qquad
\nabla_{X_2}X_2=\frac{3X_1f}{4f}X_1.
\end{array}
\end{equation}

Using Gauss formula for $M$ in $\mathbb{H}^3$, \eqref{eq2} and \eqref{eq3}, by a straightforward computation we obtain
\begin{equation}\label{eq4}
\begin{array}{c}
\nabla^{\mathbb{H}^3}_{X_1}X_1=-\frac{f}{2}\eta, \qquad
\nabla^{\mathbb{H}^3}_{X_2}X_1=-\frac{3X_1f}{4f}X_2, \qquad
\nabla^{\mathbb{H}^3}_{X_1}X_2=0, \\\\
\nabla^{\mathbb{H}^3}_{X_2}X_2=\frac{3X_1f}{4f}X_1+\frac{3f}{2}\eta
\end{array}
\end{equation}
and then, using Gauss formula for $\mathbb{H}^3$ in $\mathbb{R}^{4}_{1}$ and \eqref{eq1} we get
\begin{equation}\label{eq5}
\begin{array}{c}
\overline{\nabla}_{X_1}X_1=-\frac{f}{2}\eta +\overline{x},\qquad
\overline{\nabla}_{X_2}X_1=-\frac{3X_1f}{4f}X_2, \qquad
\overline{\nabla}_{X_1}X_2=0, \\\\
\overline{\nabla}_{X_2}X_2=\frac{3X_1f}{4f}X_1+\frac{3f}{2}\eta+\overline{x},
\end{array}
\end{equation}
where $\overline{x}\in \mathbb{H}^3$.

Now, it is easy to see that
$$
\langle \overline{\nabla}_{X_2}X_2, \overline{\nabla}_{X_2}X_2\rangle = \frac{9\left(X_1f\right)^2}{16f^2}+\frac{9f^2}{4}-1.
$$
The classification of biconservative surfaces in $\mathbb{H}^3$ will be done with respect to the sign of
$$
W=\frac{9\left(X_1f\right)^2}{16f^2}+\frac{9f^2}{4}-1=\frac{9|\grad f|^2}{16f^2}+\frac{9f^2}{4}-1.
$$
From Theorem \ref{thm:CMOP} we have that
$$
f^2=\frac{4}{3}(-1-K)>0,
$$
and, therefore
$$
\grad f=-\frac{\grad K}{\sqrt{(3(-1-K))}},
$$
where $K$ is the Gaussian curvature of $M$. It is known that the biconservative surface $\left(M^2,\varphi^\ast\langle,\rangle\right)$ is isometric to a unique abstract standard biconservative surface  $\left(D_{C_{-1}},g_{C_{-1}}\right)$ defined in the previous section. So, to $\left(M^2,\varphi^\ast\langle,\rangle\right)$ it corresponds a unique constant $C_{-1}$. Since the Gaussian curvature of $\left(D_{C_{-1}},g_{C_{-1}}\right)$ and its gradient are given in \eqref{eq:K-DC} and \eqref{eq:gradK-DC}, by a straightforward computation one obtains an equivalent expression of $W$,
$$
W=\frac{9\left|\grad K\right|^2}{64(-1-K)^2}-3K-4 =\frac{C_{-1}}{3}\xi^2, \qquad \xi\in\left(0,\xi_{01}\right).
$$
Therefore, the classification of biconservative surfaces in $\mathbb{H}^3$ will be done according the sign of the real constant $C_{-1}$.

As the gluing process will be done along the boundary given by $\xi=\xi_{01}$, we get from the above relation that we can glue only two standard biconservative surfaces corresponding to the same constant $C_{-1}$.

We denote by $\kappa_2=\sqrt{\left|W\right|}$, i.e., $\kappa_2$ is the curvature of integral curves of $X_2$. If $\kappa_2>0$, let us consider
$$
N_2=\frac{1}{\kappa_2}\left(\frac{3X_1f}{4f}X_1+\frac{3f}{2}\eta+\overline{x}\right),
$$
and if $\kappa_2=0$, i.e., $W=0$, let us denote
$$
\tilde{N}_2= \frac{3X_1f}{4f}X_1+\frac{3f}{2}\eta+\overline{x}.
$$
We note that $N_2$, $\tilde{N}_2\in C(T\mathbb{R}^{4}_{1})$, $\left|N_2\right|=1$ and $\left|\tilde{N}_2\right|=0$ . It is easy to see that $X_2f=0$ and
\begin{align*}
X_2\left(X_1f\right)= & X_1\left(X_2f\right) - \left[X_1,X_2\right]f \\
  = & \left(\overline{\nabla}_{X_1}X_2-\overline{\nabla}_{X_2}X_1\right)f \\
  = & 0.
\end{align*}
By a straightforward computation one obtains $X_2 \kappa_2=0$, i.e., the integral curves of $X_2$ are circles.

In order to compute $\overline{\nabla}_{X_1} N_2$, $\overline{\nabla}_{X_2} N_2$, $\overline{\nabla}_{X_1} \tilde{N}_2$ and $\overline{\nabla}_{X_2} \tilde{N}_2$, we first consider the orthonormal frame field $\left\{X_1,X_2, \eta, \overline{x}\right\}$ in $\mathbb{R}^{4}_{1}$ along $M$. Then, one has
$$
\overline{\nabla}_{X_1}\eta=\frac{f}{2}X_1, \qquad \overline{\nabla}_{X_2}\eta=-\frac{3f}{2}X_2, \qquad \overline{\nabla}_{X_1}N_2=0, \qquad \overline{\nabla}_{X_2}\tilde{N}_2=0,
$$
\begin{equation}\label{eq6}
\overline{\nabla}_{X_2}N_2 = \left\{
\begin{array}{ll}
\kappa_2N_2, & C_{-1}>0  \\\\
-\kappa_2N_2, & C_{-1}<0
\end{array}
\right.
\end{equation}
and
\begin{equation}\label{eq7}
\overline{\nabla}_{X_1}\tilde{N}_2 = \left(\frac{3}{4f^2}\left(X_1\left(X_1 f\right)\cdot f-\left(X_1 f\right)^2\right)+\frac{3}{4}f^2+1\right)X_1+\frac{9}{8}\left(X_1f\right)\eta+\frac{3}{4}\frac{X_1f}{f}\overline{x}.
\end{equation}

As $X_2f=0$ and $X_2\left(X_1f\right)=0$, it follows that $\overline{\nabla}_{X_2}\tilde{N}_2=0$.

Further, we consider the global problem and construct complete biconservative surfaces in $\mathbb{H}^3$ with $\grad f\neq 0$ at any point of an open dense subset by using an extrinsic approach.

We will begin with a local extrinsic characterization of biconservative surfaces in $\mathbb{H}^3$, which has been found in \cite{CMOP14}. These biconservative surfaces are called standard biconservative surfaces and, in order to reach our objective, we will glue two such surfaces.

As we have already announced, we will classify the biconservative surfaces in $\mathbb{H}^3$ with respect to the sign of the constant $C_{-1}$.

\subsection{Case $C_{-1}>0$}
$\newline$

First, we recall the local extrinsic characterization of biconservative surfaces in $\mathbb{H}^3$.

\begin{theorem}[\cite{CMOP14}]
Let $M^2$ be a biconservative surface in $\mathbb{H}^3$ with a nowhere vanishing gradient of the mean curvature function $f$. If $C_{-1}>0$, then, locally, $M^2\subset \mathbb{R}^{4}_{1}$ can be parametrized by
$$
X_{\tilde{C}_{-1}}(u,v)=\sigma(u)+\frac{4}{3\sqrt{\tilde{C}_{-1}}\kappa^{3/4}(u)}\left(c_1 \cos v +c_2 \sin v-c_1\right),
$$
where $\tilde{C}_{-1}>0$ is a positive constant; $c_1$, $c_2\in \mathbb{R}^{4}_{1}$ are two constant vectors such that $\langle c_i,c_j \rangle = \delta_{ij}$; $\sigma(u)$ is a curve parametrized by arc-length that satisfies
$$
\langle \sigma (u),c_1 \rangle = \frac{4}{3\sqrt{\tilde{C}_{-1}}\kappa^{3/4}(u)}, \qquad \langle \sigma (u),c_2 \rangle=0,
$$
so $\sigma(u)$ is a curve lying in the totally geodesic $\mathbb{H}^2=\mathbb{H}^3\cap \Pi$, where $\Pi$ is the linear hyperspace of $\mathbb{R}^{4}_{1}$ defined by $\langle \overline{r},c_2\rangle =0$, while its curvature $\kappa=\kappa(u)$ is a positive non-constant solution of the following ODE
$$
\kappa\kappa''=\frac{7}{4}\left(\kappa'\right)^2-\frac{4}{3}\kappa^2-4\kappa^4,
$$
such that
$$
\left(\kappa'\right)^2=\frac{16}{9}\kappa^2-16\kappa^4+\tilde{C}_{-1}\kappa^{7/2}.
$$
\end{theorem}

\begin{remark}
The surface $\left(M^2,X_{\tilde{C}_{-1}}^\ast\langle,\rangle\right)$ is isometric to the abstract standard biconservative surface $\left(D_{C_{-1}},g_{C_{-1}}\right)$, and the link between the constants $C_{-1}$ and $\tilde{C}_{-1}$ is
$$
C_{-1}=\frac{3^{3/4}}{16}\tilde{C}_{-1}>0.
$$
Therefore, the above parametrization $X_{\tilde{C}_{-1}}$ gives a one-parameter family of biconservative surfaces with $\grad f$ nowhere vanishing indexed by $\tilde{C}_{-1}$.
\end{remark}

\begin{remark}
We note that the biconservative surface defined by $X_{\tilde{C}_{-1}}$ is made up of circles which lie in $2$-affine planes parallel with the $2$-plane spanned  by $c_1$ and $c_2$, and which touch the curve $\sigma$. The surface is invariant under the actions of the $1$-parameter group of isometries of $\mathbb{R}^{4}_1$ with positive determinant, which acts on the $2$-plane spanned by the unit orthonormal spacelike constant vectors $c_1$ and $c_2$. In fact, the $1$-parameter group of isometries that acts on $M$ represents the flow of the Killing vector field $X_{\tilde{C}_{-1},v}$ which can be seen as a restriction to $M$ of the following Killing vector field on $\mathbb{R}^4_1$
$$
Z\left(\overline{r}\right)=-\langle \overline{r},c_2\rangle c_1+\langle \overline{r},c_1\rangle c_2,
$$
$\overline{r}$ being the position vector of a point in $\mathbb{R}^4_1$.
\end{remark}

The standard biconservative surface is ``a surface of revolution'' in $\mathbb{R}^4_1$ whose profile curve is $\sigma=\sigma(u)$ which lies in $\mathbb{R}^3_1$. We also note that the immersion $X_{\tilde{C}_{-1}}$ is, in fact, an embedding and the profile curve $\sigma$ does not have self-intersections; thus the image of $X_{\tilde{C}_{-1}}$ is a regular surface in $\mathbb{H}^3$. Therefore, in order to glue two standard biconservative surfaces in $\mathbb{R}^4_{1}$, it is enough to glue two profile curves defining them, in this way obtaining a complete biconservative regular surface in $\mathbb{H}^3$.

Our strategy is as follows: we reparametrize the profile curve $\sigma$ in a more convenient way and get $\sigma=\sigma(\kappa)$, then, since the gluing process of the curves $\sigma$ implies all its components  (three components) it is more convenient to chose another model for $\mathbb{H}^3$ (the upper half space) such that, after that transformation, the curve $\sigma$ would have two components. After the gluing process is performed, we will obtain a closed regular curve in the upper half plane and therefore, we will get a closed biconservative regular surface in $\mathbb{H}^3$ which has to be complete.

In the same paper \cite{CMOP14}, it was proved that $\kappa_2=\left(3\sqrt{\tilde{C}_{-1}}\kappa^{3/4}\right)/4$ where $\kappa$ is the geodesic curvature of $\sigma$ in $\mathbb{H}^2$  given by $\kappa(u)=f(u)/2$ and $\kappa_2$ is the curvature of integral curves of $X_2$, using the same notations as in the previous section.

Choosing $c_1=e_1$ and $c_2=e_2$, where $\left\{e_1,e_2,e_3,e_4\right\}$ is the canonical basis of $\mathbb{R}^4_1$, the curve $\sigma$ can be rewritten as
$$
\sigma(u)=\left(\frac{4}{3\sqrt{\tilde{C}_{-1}}\kappa^{3/4}(u)},0,x(u),y(u)\right),
$$
for some functions $x=x(u)$ and $y=y(u)$ which are solutions of the following system
\begin{equation}\label{eq:system1}
\left\{
\begin{array}{l}
\frac{16}{9\tilde{C}_{-1}\kappa^{3/2}(u)}+x^2(u)-y^2(u)=-1 \\\\
\frac{16\left(1-9\kappa^2(u)\right)}{9\tilde{C}_{-1}\kappa^{3/2}(u)}+x'^2(u)-y'^2(u)=0 \\\\
\frac{16\left(1+3\kappa^2(u)\right)^2}{9\tilde{C}_{-1}\kappa^{3/2}(u)}+x''^2(u)-y''^2(u)=\kappa^2(u)-1
\end{array}
\right..
\end{equation}
These equations are obtained from the relations $\sigma(u)\in\mathbb{H}^3$, $\left| \sigma'(u)\right|^2=1$ and $\left| \sigma''(u)\right|^2=\kappa^2(u)-1$, for any $u$. We can also assume that the function $y$ is a positive function as $\sigma$ lies in $\mathbb{H}^3$.

In order to prove that there exists a curve $\sigma$ satisfying \eqref{eq:system1}, let us consider the change of coordinates
$$
\left\{
\begin{array}{c}
  x(u)=R(u)\sinh(\mu(u)) \\\\
  y(u)=R(u)\cosh(\mu(u))
\end{array}
\right.,
$$
with $R(u)>0$ and $\mu(u)\in (0,2\pi)$, for any $u$ in an open interval $I$.

Then, from the first equation of \eqref{eq:system1}, one obtains
\begin{equation}\label{eq:R}
R(u)=\frac{\sqrt{9\tilde{C}_{-1}\kappa^{3/2}(u)+16}}{3\sqrt{\tilde{C}_{-1}}\kappa^{3/4}(u)}.
\end{equation}
Since $\kappa(u)>0$, for any $u$, we can think $u=u(\kappa)$, so $R=R(\kappa)$, $\mu=\mu(\kappa)$ and then, by a straightforward computation, we get from the second equation of \eqref{eq:system1}
\begin{equation}\label{eq:muOfk}
\mu(\kappa)=\pm\int_{\kappa_{00}}^{\kappa}\frac{36\sqrt{\tilde{C}_{-1}}\tau^{7/4}}{\left(9\tilde{C}_{-1}\tau^{3/2}+16\right)\sqrt{\frac{16}{9}\tau^2-16\tau^4+\tilde{C}_{-1}\tau^{7/2}}} \ d\tau +c_{0}, \qquad c_0\in\mathbb{R},
\end{equation}
for any $\kappa\in\left(0,\kappa_{01}\right)$, where $\kappa_{01}$ is the positive vanishing point of $16\kappa^2/9-16\kappa^4+\tilde{C}_{-1}\kappa^{7/2}$, $16\kappa^2/9-16\kappa^4+\tilde{C}_{-1}\kappa^{7/2}>0$, for any $\kappa\in\left(0,\kappa_{01}\right)$, $\kappa_{01}>\left(3\tilde{C}_{-1}\right)^2/2^{12}$, and $\kappa_{00}$ is arbitrarily fixed in $\left(0,\kappa_{01}\right)$.

Now, it is easy to see that the first two equations of \eqref{eq:system1} imply the third one.

\begin{remark}
As the classification of biconservative surfaces is done up to isometries of $\mathbb{H}^3$, the sign of the integral and the constant $c_{0}$ in the expression of $\mu$ play an active role only in the gluing process.
\end{remark}

The following lemma can be easily proved using similar arguments as in the proof of Lemma \ref{lemma1}.
\begin{lemma}
Consider
$$
\mu_0(\kappa)=\int_{\kappa_{00}}^{\kappa}\frac{36\sqrt{\tilde{C}_{-1}}\tau^{7/4}}{\left(9\tilde{C}_{-1}\tau^{3/2}+16\right)\sqrt{\frac{16}{9}\tau^2-16\tau^4+\tilde{C}_{-1}\tau^{7/2}}} \ d\tau,
$$
i.e., we fix the sign in \eqref{eq:muOfk} and we choose $c_0=0$. Then
\begin{itemize}
  \item [(i)] $\lim_{\kappa\searrow 0}\mu_0(\kappa)=\mu_{0,-1}\in (-\infty,0)$ and $\lim_{\kappa\nearrow \kappa_{01}}\mu_0(\kappa)=\mu_{0,1}\in (0,\infty)$.
  \item [(ii)] $\mu_0$ is strictly increasing,
  $$
  \lim_{\kappa\searrow 0}\mu'_0(\kappa)=0 \text{ \quad and \quad }\lim_{\kappa\nearrow \kappa_{01}}\mu'_0(\kappa)=\infty.
  $$
  \item [(iii)] $\lim_{\kappa\searrow 0}\mu''_0(\kappa)=0$ and $\lim_{\kappa\nearrow \kappa_{01}}\mu''_0(\kappa)=-\infty$.
\end{itemize}
\end{lemma}

Now, we can give an explicit expression of the profile curve $\sigma$:
$$
\sigma(\kappa)=\left(\frac{4}{3\sqrt{\tilde{C}_{-1}}\kappa^{3/4}}, R(\kappa)\sinh \mu(\kappa), R(\kappa)\cosh \mu(\kappa)\right),
$$
for any $\kappa\in\left(0,\kappa_{01}\right)$, where $\mu$ is given in equation \eqref{eq:muOfk} and
$$
R(\kappa)=\frac{\sqrt{9\tilde{C}_{-1}\kappa^{3/2}+16}}{3\sqrt{\tilde{C}_{-1}}\kappa^{3/4}}.
$$
Therefore, we can rewrite
\begin{equation}\label{eq:localSurface1}
X_{\tilde{C}_{-1}}(\kappa,v)=\left(\frac{4\cos v}{3\sqrt{\tilde{C}_{-1}}\kappa^{3/4}},\frac{4\sin v}{3\sqrt{\tilde{C}_{-1}}\kappa^{3/4}}, R(\kappa)\sinh \mu(\kappa),R(\kappa)\cosh \mu(\kappa)\right),
\end{equation}
for any $(\kappa,v)\in \left(0,\kappa_{01}\right)\times\mathbb{R}$.

In order to obtain a complete biconservative surface in $\mathbb{H}^3$, first we reparametrize the profile curve $\sigma$ as we announced earlier.
Through the standard diffeomorphism from hyperboloid model to upper half space model.
\begin{equation}\label{eq:stand_diff}
\delta \left(x^1,x^2,x^3,x^4\right)=\left(1, \frac{2x^2}{x^1+x^4},\frac{2x^3}{x^1+x^4},\frac{2}{x^1+x^4}\right)
\end{equation}
the profile curve $\sigma$ becomes
\begin{align*}
  \sigma(\kappa)= & \left(1,\frac{2\sqrt{9\tilde{C}_{-1}\kappa^{3/2}+16}\sinh \mu(\kappa)}{4+\sqrt{9\tilde{C}_{-1}\kappa^{3/2}+16}\cosh \mu(\kappa)}, \frac{6\sqrt{\tilde{C}_{-1}}\kappa^{3/4}}{4+\sqrt{9\tilde{C}_{-1}\kappa^{3/2}+16}\cosh \mu(\kappa)}\right) \\
  \equiv & \left(\frac{2\sqrt{9\tilde{C}_{-1}\kappa^{3/2}+16}\sinh \mu(\kappa)}{4+\sqrt{9\tilde{C}_{-1}\kappa^{3/2}+16}\cosh \mu(\kappa)}, \frac{6\sqrt{\tilde{C}_{-1}}\kappa^{3/4}}{4+\sqrt{9\tilde{C}_{-1}\kappa^{3/2}+16}\cosh \mu(\kappa)}\right).
\end{align*}
Choosing appropriate values of the constant $c_0$ and of the sign in the expression of $\mu(k)$, we can find two profile curves $\sigma_1$ and $\sigma_2$ such that we can glue them smoothly.

So, let us consider the following two curves
$$
\sigma_1(\kappa)= \frac{\left(2\sqrt{9\tilde{C}_{-1}\kappa^{3/2}+16}\sinh \mu_0(\kappa),6\sqrt{\tilde{C}_{-1}}\kappa^{3/4}\right)}{4+\sqrt{9\tilde{C}_{-1}\kappa^{3/2}+16}\cosh \mu_0(\kappa)},
$$
i.e., we take, for the sake of simplicity, $c_0=0$ and the sign ``+'' in \eqref{eq:muOfk},
and then
\begin{align*}
  \sigma_2(\kappa)= & \frac{\left(2\sqrt{9\tilde{C}_{-1}\kappa^{3/2}+16}\sinh\left(-\mu_0(\kappa)+2\mu_{0,1}\right),
   6\sqrt{\tilde{C}_{-1}}\kappa^{3/4}\right)}{4+\sqrt{9\tilde{C}_{-1}\kappa^{3/2}+16}\cosh \left(-\mu_0(\kappa)+2\mu_{0,1}\right)},
\end{align*}
i.e., we take $c_0=2\mu_{0,1}$ and the sign ``-'' in \eqref{eq:muOfk}. If we choose a different value for $c_0$ and a different sign in \eqref{eq:muOfk}, then the new curve $\sigma_2$ cannot be glued at the $C^1$ smoothness level with $\sigma_1$.

It is easy to see that
$$
\lim_{\kappa\nearrow \kappa_{01}} \sigma_1(\kappa)=\lim_{\kappa\nearrow \kappa_{01}} \sigma_2(\kappa)=\left(\frac{6\kappa_{01}\sinh \mu_{0,1}}{1+3\kappa_{01}\cosh\mu_{0,1}},\frac{3\sqrt{\tilde{C}_{-1}}\kappa_{01}^{3/4}}{2+6\kappa_{01}\cosh \mu_{0,1}}\right)\in\mathbb{R}^2.
$$
Let us denote by
$$
x_1(\kappa)=\frac{2\sqrt{9\tilde{C}_{-1}\kappa^{3/2}+16}\sinh \mu(\kappa)}{4+\sqrt{9\tilde{C}_{-1}\kappa^{3/2}+16}\cosh \mu(\kappa)},\qquad y_1(\kappa)=\frac{6\sqrt{\tilde{C}_{-1}}\kappa^{3/4}}{4+\sqrt{9\tilde{C}_{-1}\kappa^{3/2}+16}\cosh \mu(\kappa)}.
$$
It is easy to see that $x_1'(\kappa)>0$ in a neighborhood of $\kappa_{01}$, $\left(\kappa_{01}-\varepsilon,\kappa_{01}\right)$. If we denote by $x_{0,1}=\lim_{\kappa\nearrow \kappa_{01}}x_1(\kappa)$ and $x_{0,-1}=\lim_{\kappa\searrow \kappa_{01}-\varepsilon}x_1(\kappa)$, it follows that there exists the inverse function $x_1^{-1}=\kappa_1:\left(x_{0,-1},x_{0,1}\right)\to \left(\kappa_{01}-\varepsilon,\kappa_{01}\right)$.

So, we can consider $y_1=y_1(x)$ and by a straightforward computation one obtains
\begin{equation*}
\lim_{x\nearrow x_{0,1}}\frac{dy_1}{dx}=-\frac{3\sqrt{\tilde{C}_{-1}}\kappa_{01}^{3/4}\sinh \mu_{0,1}}{4\cosh \mu_{0,1}+12 \kappa_{01}}.
\end{equation*}
Similarly, we denote by
\begin{equation*}
x_2(\kappa)=\frac{2\sqrt{9\tilde{C}_{-1}\kappa^{3/2}+16}\sinh\left(-\mu_0(\kappa)+2\mu_{0,1}\right)}{4+\sqrt{9\tilde{C}_{-1}\kappa^{3/2}+16}\cosh \left(-\mu_0(\kappa)+2\mu_{0,1}\right)},
\end{equation*}
and
\begin{equation*}
y_2(\kappa)=\frac{6\sqrt{\tilde{C}_{-1}}\kappa^{3/4}}{4+\sqrt{9\tilde{C}_{-1}\kappa^{3/2}+16}\cosh \left(-\mu_0(\kappa)+2\mu_{0,1}\right)}.
\end{equation*}
It is easy to see that $x_2'(\kappa)<0$ in a neighborhood of $\kappa_{01}$, $\left(\kappa_{01}-\varepsilon,\kappa_{01}\right)$. Since $\lim_{\kappa\nearrow \kappa_{01}}x_2(\kappa)=x_{0,1}$ and denoting by $x_{1,-1}=\lim_{\kappa\searrow \kappa_{01}-\varepsilon}x_2(\kappa)$, it follows that there exists the inverse function $x_2^{-1}=\kappa_2:\left(x_{0,1},x_{1,-1}\right)\to \left(\kappa_{01}-\varepsilon,\kappa_{01}\right)$.

So, we can also consider $y_2=y_2(x)$ and to glue $\sigma_1$ and $\sigma_2$ at the $C^1$ smoothness level means that $y_1=y_1(x)$ and $y_2=y_2(x)$ yield a $C^1$ smooth function around $x_{0,1}$. We note that this is equivalent to the fact that $\sigma_1$ and $\sigma_2$ have the same tangent space at the gluing point.

By a straightforward computation one obtains
\begin{equation*}
\lim_{x\searrow x_{0,1}}\frac{dy_2}{dx}=\lim_{x\nearrow x_{0,1}}\frac{dy_1}{dx}=-\frac{3\sqrt{\tilde{C}_{-1}}\kappa_{01}^{3/4}\sinh \mu_{0,1}}{4\cosh \mu_{0,1}+12 \kappa_{01}}\in\mathbb{R}.
\end{equation*}
One can also show that the second and the third derivative of $y_1$ exist at $x_{0,1}$.

Thus, gluing the two curves $\sigma_1$ and $\sigma_2$, we obtain at least a $C^3$ smooth curve. Moreover, the curve obtained by the gluing process is a closed regular curve in the upper half plane and therefore, we get a closed biconservative regular surface in $\mathbb{H}^3$ which has to be complete. In Figure $1$ we represent the curves $\sigma_1$  and $\sigma_2$, with the colors red and blue, respectively, for the constant $\tilde{C}_{-1}=1$, and in Figure $2$ we represent the corresponding surfaces to $\sigma_1$  and $\sigma_2$ in the upper half space (with the Euclidean metric).

\begin{tabular}{c c}
\includegraphics[width=0.4\textwidth]{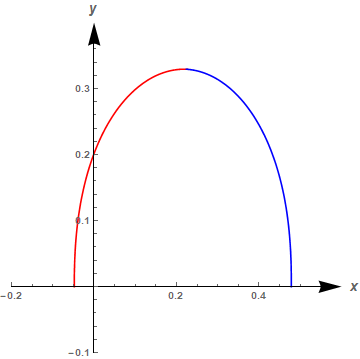}
&
\includegraphics[width=0.7\textwidth]{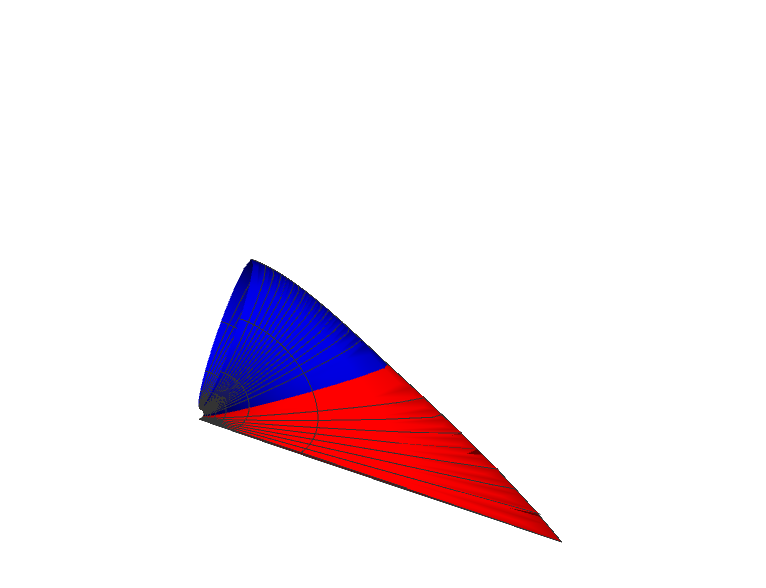}
\\
Figure $1$. The profile curves
&
Figure $2$. The corresponding surfaces
\\
$\sigma_1$  and $\sigma_2$
&
to $\sigma_1$  and $\sigma_2$
\end{tabular}

\begin{remark}
If we denote by $S^{\pm}_{\tilde{C}_{-1},c_0}$ the image of $X_{\tilde{C}_{-1}}$ (given in \eqref{eq:localSurface1}) corresponding to the sign ``+'' or ``-'' in \eqref{eq:muOfk}, we obtain $S^{-}_{\tilde{C}_{-1},2\mu_{0,1}}=T\left(S^{+}_{\tilde{C}_{-1},0}\right)$, where $T$ is the symmetry of $\mathbb{R}^4_1$ with respect to the $\left(Ox^1x^2x^4\right)$ hyperplane followed by a linear orthogonal transformation with positive determinant. Moreover, the boundary of $S^{+}_{\tilde{C}_{-1},0}$ is given by the circle
$$
v\to\ \left( \frac{4\cos v}{3\sqrt{\tilde{C}_{-1}}\kappa^{3/4}_{01}},\frac{4\sin v}{3\sqrt{\tilde{C}_{-1}}\kappa^{3/4}_{01}}, R\left(\kappa_{01}\right)\sinh \mu_{0,1},R\left(\kappa_{01}\right)\cosh \mu_{0,1}\right),
$$
and this circle is invariant by $T$.
\end{remark}

\begin{remark}
Now it is not difficult to give the explicit expression of the biconservative immersion $F$ in Theorem \ref{th:EUH3}. Moreover, if we want to use Theorem \ref{th:EUH3} in the gluing process, then it is enough to assure that the gluing process in at least of $C^1$ smoothness as, once the first standard biconservative surface is fixed, the $C^1$ gluing process determines uniquely the second standard biconservative surface (the sign ``+'' or ``-'' and the constant $c_0$ in \eqref{eq:muOfk}). As the two standard biconservative surfaces give $F$, the gluing process is in fact at least of the class $C^3$.
\end{remark}

\subsection{Case $C_{-1}<0$}
$\newline$

First, we recall a local extrinsic result which provides a characterization of biconservative surfaces in $\mathbb{H}^3$ when $C_{-1}<0$.

\begin{theorem}[\cite{CMOP14}]
Let $M^2$ be a biconservative surface in $\mathbb{H}^3$ with $\grad f$ nowhere vanishing. If $C_{-1}<0$, then, locally, $M^2\subset \mathbb{R}^{4}_{1}$ can be parametrized by
$$
X_{\tilde{C}_{-1}}(u,v)=\sigma(u)+\frac{4}{3\sqrt{-\tilde{C}_{-1}}\kappa^{3/4}(u)}\left(c_1\sinh v+c_2\cosh v-c_2\right),
$$
where $\tilde{C}_{-1}<0$ is a negative constant; $c_1$, $c_2\in \mathbb{R}^{4}_{1}$ are two constant vectors such that $\langle c_1,c_1 \rangle=1$, $\langle c_2,c_2 \rangle=-1$ and $\langle c_1,c_2 \rangle =0$; $\sigma(u)$ is a curve parameterized by arc-length that satisfies
$$
\langle \sigma (u),c_1 \rangle =0,\qquad \langle \sigma (u),c_2 \rangle= -\frac{4}{3\sqrt{-\tilde{C}_{-1}}\kappa^{3/4}(u)},
$$
so $\sigma(u)$ is a curve lying in the totally geodesic $\mathbb{H}^2=\mathbb{H}^3\cap \Pi$, where $\Pi$ is the linear hyperspace of $\mathbb{R}^{4}_{1}$ defined by $\langle \overline{r},c_1\rangle =0$, while its curvature $\kappa=\kappa(u)$ is a positive non-constant solution of the following ODE
$$
\kappa\kappa''=\frac{7}{4}\left(\kappa'\right)^2-\frac{4}{3}\kappa^2-4\kappa^4,
$$
such that
$$
\left(\kappa'\right)^2=\frac{16}{9}\kappa^2-16\kappa^4+\tilde{C}_{-1}\kappa^{7/2}.
$$
\end{theorem}

\begin{remark}
The surface $\left(M^2,X_{\tilde{C}_{-1}}^\ast\langle,\rangle\right)$ is isometric to the abstract standard biconservative surface $\left(D_{C_{-1}},g_{C_{-1}}\right)$, and the link between the constants $C_{-1}$ and $\tilde{C}_{-1}$ is
$$
C_{-1}=\frac{3^{3/4}}{16}\tilde{C}_{-1}<0.
$$
\end{remark}

\begin{remark}
We note that the biconservative surface defined by $X_{\tilde{C}_{-1}}$ is made up of (branches of) ``hyperbolas'' which lie in $2$-affine planes parallel with the $2$-plane spanned  by $c_1$ and $c_2$, and whose vertices belong to the curve $\sigma$. The surface is invariant under the actions of the $1$-parameter group of isometries of $\mathbb{R}^{4}_1$ with positive determinant, which acts on the Minkowski $2$-plane spanned by the constant vectors $c_1$ and $c_2$. In fact, the $1$-parameter group of isometries that acts on $M$ represents the flow of the Killing vector field $X_{\tilde{C}_{-1},v}$ which can be seen as a restriction to $M$ of the following Killing vector field on $\mathbb{R}^4_1$
$$
Z\left(\overline{r}\right)=-\langle \overline{r},c_2\rangle c_1+\langle \overline{r},c_1\rangle c_2,
$$
$\overline{r}$ being the position vector of a point in $\mathbb{R}^4_1$.
\end{remark}

\begin{remark}
We note that we slightly corrected the expression of the above local parametrization: the multiplicative coefficient in formula $(44)$, in the original paper \cite{CMOP14}, should be $2\sqrt{2}/\left(3\sqrt{-C}k(u)^{3/4}\right)$ and not $4/\left(3\sqrt{-C}k(u)^{3/4}\right)$.
\end{remark}

As in the first case, it can be shown that $\kappa_2=\left(3\sqrt{-\tilde{C}_{-1}}\kappa^{3/4}\right)/4$ where $\kappa$ is the geodesic curvature of $\sigma$ in $\mathbb{H}^2$  given by $\kappa(u)=f(u)/2$ and $\kappa_2$ is the curvature of integral curves of $X_2$. Choosing
$$
c_1=e_2,\qquad c_2=e_1+\sqrt{2}e_4,
$$
where $\left\{e_1,e_2,e_3,e_4\right\}$ is the canonical basis of $\mathbb{R}^4_1$, the curve $\sigma$ can be rewritten as
$$
\sigma(u)=\left(\sqrt{2}y(u)-\frac{4}{3\sqrt{-\tilde{C}_{-1}}\kappa^{3/4}(u)},0,x(u),y(u)\right),
$$
for some functions $x=x(u)$ and $y=y(u)$ that have to be solutions of the system
\begin{equation}\label{eq:system2}
\left\{
\begin{array}{l}
\left(y(u)-\frac{4\sqrt{2}}{3\sqrt{-\tilde{C}_{-1}}\kappa^{3/4}(u)}\right)^2+x^2(u)=-1+\frac{16}{-9\tilde{C}_{-1}\kappa^{3/2}(u)}, \qquad y(u)>0 \\\\
\left(y(u)-\frac{4\sqrt{2}}{3\sqrt{-\tilde{C}_{-1}}\kappa^{3/4}(u)}\right)'^2+x'^2(u)=1+\left(\frac{4}{3\sqrt{-\tilde{C}_{-1}}\kappa^{3/4}(u)}\right)'^2 \\\\
\left(y(u)-\frac{4\sqrt{2}}{3\sqrt{-\tilde{C}_{-1}}\kappa^{3/4}(u)}\right)''^2+x''^2(u)=\kappa^2(u)-1+\left(\frac{4}{3\sqrt{-\tilde{C}_{-1}}\kappa^{3/4}(u)}\right)''^2.
\end{array}
\right.
\end{equation}

In order to prove that there exists a curve $\sigma$ which satisfies \eqref{eq:system2}, let us consider the change of coordinates
$$
\left\{
\begin{array}{l}
  x(u)=R(u)\cos(\mu(u)) \\\\
  y(u)=R(u)\sin(\mu(u))+\frac{4\sqrt{2}}{3\sqrt{-\tilde{C}_{-1}}\kappa^{3/4}(u)}
\end{array}
\right.,
$$
with $R(u)>0$ and $\mu(u)\in (0,2\pi)$, for any $u$ in an open interval $I$.

Then, from the first equation of \eqref{eq:system2}, one obtains
\begin{equation}\label{eq:R2}
R(u)=\frac{\sqrt{9\tilde{C}_{-1}\kappa^{3/2}(u)+16}}{3\sqrt{-\tilde{C}_{-1}}\kappa^{3/4}(u)}.
\end{equation}
Further, it is easy to see that $16\kappa^2/9-16\kappa^4+\tilde{C}_{-1}\kappa^{7/2}>0$, for any $\kappa\in\left(0,\kappa_{01}\right)$ where $\kappa_{01}>0$ is the positive vanishing point of the function $16\kappa^2/9-16\kappa^4+\tilde{C}_{-1}\kappa^{7/2}$, for $\tilde{C}_{-1}$ a negative fixed scalar. Then, it is clear that $16+9\tilde{C}_{-1}\kappa^{3/2}>0$, for any $\kappa\in\left(0,\kappa_{01}\right)$. Using this inequality and \eqref{eq:R2} it is possible to verify that the change of coordinates is correct, i.e., $R(u)\sin(\mu(u))+\left(4\sqrt{2}\right)/\left(3\sqrt{-\tilde{C}_{-1}}\kappa^{3/4}(u)\right)>0$, for any $u\in I$.

In this case, our strategy is similar to that used in the previous case: in order to obtain a complete biconservative surface in $\mathbb{H}^3$, we will glue two standard biconservative surfaces and for this it is enough to glue the two ``profile curves'' defining them.  We will obtain a closed regular curve in the upper half plane, so a closed biconservative surface in $\mathbb{H}^3$.

First, we reparametrize the profile curve $\sigma$ in a more convenient way and choose the appropriate model for $\mathbb{H}^3$ (the upper half space). Since $\kappa(u)>0$, for any $u$, we can consider $u=u(\kappa)$, so $R=R(\kappa)$, $\mu=\mu(\kappa)$ and then, by a similar computation as in the $C_{-1}>0$ case, we get

\begin{equation}\label{eq:muOfk2}
\mu(\kappa)=\pm\int_{\kappa_{00}}^{\kappa}\frac{36\sqrt{-\tilde{C}_{-1}}\tau^{7/4}}{\left(9\tilde{C}_{-1}\tau^{3/2}+16\right)\sqrt{\frac{16}{9}\tau^2-16\tau^4+\tilde{C}_{-1}\tau^{7/2}}} \ d\tau +c_{0}, \qquad c_0\in\mathbb{R},
\end{equation}
for any $\kappa\in\left(0,\kappa_{01}\right)$ and $\kappa_{00}$ arbitrarily fixed in $\left(0,\kappa_{01}\right)$, where $\kappa_{01}$ is the vanishing point of $16\kappa^2/9-16\kappa^4+\tilde{C}_{-1}\kappa^{7/2}$.

Further, we denote by
$$
\mu_0(\kappa)=\int_{\kappa_{00}}^{\kappa}\frac{36\sqrt{-\tilde{C}_{-1}}\tau^{7/4}}{\left(9\tilde{C}_{-1}\tau^{3/2}+16\right)\sqrt{\frac{16}{9}\tau^2-16\tau^4+\tilde{C}_{-1}\tau^{7/2}}} \ d\tau,
$$
and then $\mu(\kappa)=\pm\mu_0(\kappa)+c_0$. We will also preserve the same notations for the limits of $\mu_0$ in $0$ and in $\kappa_{01}$, i.e.,
$$
\mu_{0,-1}=\lim_{\kappa\searrow 0}\mu_0(\kappa),\qquad \mu_{0,1}=\lim_{\kappa\nearrow \kappa_{01}}\mu_0(\kappa),
$$
where $\mu_{0,-1}\in (-\infty,0)$ and $\mu_{0,1}\in (0,\infty)$.

Thus, the explicit expression of the profile curve $\sigma$ is
{\small{
$$
\sigma(\kappa)=\left(\sqrt{2}R(\kappa)\sin \mu(\kappa)+\frac{4}{3\sqrt{-\tilde{C}_{-1}}\kappa^{3/4}}, R(\kappa)\cos \mu(\kappa), R(\kappa)\sin \mu(\kappa)+\frac{4\sqrt{2}}{3\sqrt{-\tilde{C}_{-1}}\kappa^{3/4}}\right),
$$
}}
for any $\kappa\in\left(0,\kappa_{01}\right)$, where $\mu$ is given in equation \eqref{eq:muOfk2} and
$$
R(\kappa)=\frac{\sqrt{9\tilde{C}_{-1}\kappa^{3/2}+16}}{3\sqrt{-\tilde{C}_{-1}}\kappa^{3/4}}.
$$
and therefore,
\begin{eqnarray}\label{eq:localSurface2} \nonumber
  X_{\tilde{C}_{-1}}(\kappa,v) &=&  \left(\sqrt{2}R(\kappa)\sin \mu(\kappa)+\frac{4\cosh v} {3\sqrt{-\tilde{C}_{-1}}\kappa^{3/4}}, \frac{4\sinh v}{3\sqrt{-\tilde{C}_{-1}}\kappa^{3/4}},\right. \\
   & &  \left.\quad  R(\kappa)\cos \mu(\kappa), R(\kappa)\sin \mu(\kappa)+\frac{4\sqrt{2}\cosh v}{3\sqrt{-\tilde{C}_{-1}}\kappa^{3/4}}\right),
\end{eqnarray}
for any $(\kappa,v)\in \left(0,\kappa_{01}\right)\times\mathbb{R}$.

Through the standard diffeomorphism from hyperboloid model to upper half space model \eqref{eq:stand_diff}, the profile curve $\sigma$ becomes
$$
\sigma(\kappa)=\frac{\left(2\sqrt{9\tilde{C}_{-1}\kappa^{3/2}+16}\cos \mu(\kappa), 6\sqrt{-\tilde{C}_{-1}}\kappa^{3/4}\right)}{{\left(1+\sqrt{2}\right)\left(4+\sqrt{9\tilde{C}_{-1}\kappa^{3/2}+16}\sin \mu(\kappa)\right)}}.
$$
In order to find two curves that we will glue such that the gluing process to be smooth (in fact, at least of class $C^3$), we make the same choices of the constant $c_0$ and of the sign in \eqref{eq:muOfk2}, as in $C_{-1}>0$ case, so let us consider
$$
\sigma_1(\kappa)= \frac{\left(2\sqrt{9\tilde{C}_{-1}\kappa^{3/2}+16}\cos \mu_0(\kappa), 6\sqrt{-\tilde{C}_{-1}}\kappa^{3/4}\right)}{\left(1+\sqrt{2}\right)\left(4+\sqrt{9\tilde{C}_{-1}\kappa^{3/2}+16}\sin \mu_0(\kappa)\right)}
$$
and
$$
\sigma_2(\kappa)= \frac{\left(2\sqrt{9\tilde{C}_{-1}\kappa^{3/2}+16}\cos \left(-\mu_0(\kappa)+2\mu_{0,1}\right),6\sqrt{-\tilde{C}_{-1}}\kappa^{3/4}\right)}{\left(1+\sqrt{2}\right)\left(4+\sqrt{9\tilde{C}_{-1}\kappa^{3/2}+16}\sin \left(-\mu_0(\kappa)+2\mu_{0,1}\right)\right)}.
$$
By computations similar to the first case, we can see that gluing the two curves $\sigma_1$ and $\sigma_2$, one obtains at least a $C^3$ smooth curve. Moreover, this is a closed regular curve in the upper half plane, so the corresponding biconservative surface is a closed one in $\mathbb{H}^3$. In Figure $3$ we represent the curves $\sigma_1$  and $\sigma_2$, with the colors red and blue, respectively, for the constant $\tilde{C}_{-1}=-1$, and in Figure $4$ we represent the corresponding surfaces to $\sigma_1$  and $\sigma_2$ in the upper half space (with the Euclidean metric).

\begin{tabular}{c c}
\includegraphics[width=0.5\textwidth]{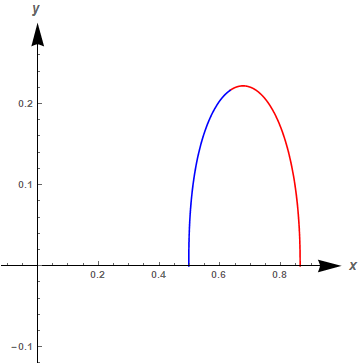}
&
\includegraphics[width=0.4\textwidth]{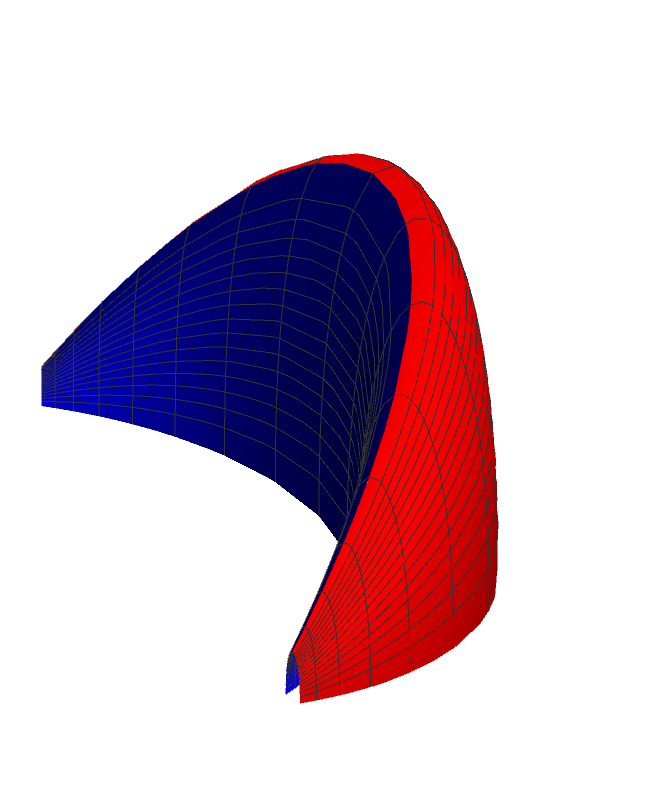}
\\
Figure $3$. The profile curves
&
Figure $4$. The corresponding surfaces
\\
$\sigma_1$  and $\sigma_2$
&
to $\sigma_1$  and $\sigma_2$
\end{tabular}

\subsection{Case $C_{-1}=0$}
$\newline$
We mention that a local extrinisic characterization of biconservative surfaces when $C_{-1}=0$ is given in \cite{F15}. However, using a similar technique as in \cite{CMOP14}, we prove the following local extrinsic result concerning biconservative surfaces in $\mathbb{H}^3$ when $C_{-1}=0$.

\begin{theorem}
Let $M^2$ be a biconservative surface in $\mathbb{H}^3$ with a nowhere vanishing gradient of the mean curvature function $f$. If $C_{-1}=0$, then, locally, $M^2\subset \mathbb{R}^{4}_{1}$ can be parametrized by
$$
X(u,v)=\sigma(u)+2^{3/4}\kappa^{3/4}(u)v^2c_1+vc_2,
$$
where $c_1$, $c_2\in \mathbb{R}^{4}_{1}$ are two constant vectors such that $\langle c_1,c_1\rangle=\langle c_1,c_2\rangle=0$, $\langle c_2,c_2 \rangle=1$; $\sigma(u)$ is a curve parameterized by arc-length that satisfies
$$
\langle \sigma(u),c_1\rangle=-\frac{1}{2^{7/4}\kappa^{3/4}(u)}, \qquad \langle \sigma (u),c_2\rangle=0,
$$
so $\sigma(u)$ is a curve lying in the totally geodesic $\mathbb{H}^2=\mathbb{H}^3\cap \Pi$, where $\Pi$ is the linear hyperspace of $\mathbb{R}^{4}_{1}$ defined by $\langle \overline{r},c_2\rangle =0$, while its curvature $\kappa=\kappa(u)$ is a positive non-constant solution of the following ODE
$$
\kappa\kappa''=\frac{7}{4}\left(\kappa'\right)^2-\frac{4}{3}\kappa^2-4\kappa^4,
$$
such that
\begin{equation}\label{eq:firstIntegralW=0}
\left(\kappa'\right)^2=\frac{16}{9}\kappa^2-16\kappa^4.
\end{equation}
\end{theorem}

\begin{proof}
Let $\gamma$ be an integral curve of $X_2$ parametrized by arc-length and
$$
\tilde{N}_2=\frac{3X_1f}{4f}X_{1}+\frac{3f}{2}\eta+\overline{x},
$$
where $X_1$, $X_2$ and $\tilde{N}_2$ are defined in the previous section. Then, as
$$
\gamma''(s)=\overline{\nabla}_{\gamma'}\gamma'=\tilde{N}_2(s),
$$
where $\tilde{N}_2(s)=\tilde{N}_2(\gamma(s))$, and $\overline{\nabla}_{X_2}\tilde{N}_2=0$, it follows that
$$
\gamma'''(s)= \overline{\nabla}_{\gamma'}\gamma''=0.
$$
Therefore, $\gamma$ can be parametrized by
$$
\gamma(s)=\frac{1}{2}a_{1}s^2+a_2 s+a_3, \qquad a_1,a_2,a_3\in\mathbb{R}^4_1
$$
with
$$
\langle a_1, a_1\rangle=0, \qquad \langle a_2, a_2\rangle=1, \qquad \langle a_1,a_2\rangle=0.
$$
Let $p_0\in M$ be an arbitrary point, and $\sigma=\sigma(u)$ be an integral curve of $X_1$, with $\sigma(0)=p_0$. Considering the flow $\phi$ of the vector field $X_2$ around the point $p_0$, i.e., $\phi'_{\sigma(u)}(v)=X_2\left(\phi_{\sigma(u)}(v)\right)$, we obtain
$$
\phi_{\sigma(u)}(s)= \frac{1}{2}a_{1}(u)s^2+a_2(u) s+a_3(u),
$$
for any $u\in(-\delta,\delta)$ and for any $s\in(-\varepsilon,\varepsilon)$, where the functions $a_1$, $a_2$ and $a_3$ satisfy
$$
\langle a_1(u), a_1(u)\rangle=0, \qquad \langle a_2(u), a_2(u)\rangle=1, \qquad \langle a_1(u),a_2(u)\rangle=0, \qquad u\in(-\delta,\delta).
$$
Then, the surface can be locally parametrized by
$$
X(u,s)=\phi_{\sigma(u)}(s).
$$
Since $X(u,0)= \phi_{\sigma(u)}(0)$, it is clear that $a_3(u)=\sigma(u)$. Moreover, from $\phi'_{\sigma(u)}(0)=X_2\left(\phi_{\sigma(u)}(0)\right)$, we have $a_2(u)=X_2(u)$, where $X_2(u)=X_2(\sigma(u))$.

We know that $\phi''_{\sigma(u)}(0)=a_1(u)$ and $\gamma''(0)=\tilde{N}_2(\phi_{\sigma(u)}(0))=\tilde{N}_2(\sigma(u))$. Then, we get $a_1(u)=\tilde{N}_2(u)$, where $\tilde{N}_2(u)=\tilde{N}_2(\sigma(u))$.

Now, we can see that $a_2(u)$ is, in fact, a constant vector.  Indeed, as $\overline{\nabla}_{X_1}X_2=0$ (from \eqref{eq5}) and as $\sigma$ is an integral curve of $X_1$, it follows that $a_2$ is a constant vector, i.e., it does not depend on $u$.

In order to see that $a_1$ is constant, we use the fact that $W=0$, i.e.,
$$
\left(X_1f\right)^2=\frac{16}{9}f^2\left(1-\frac{9}{4}f^2\right).
$$
and therefore $f\in (0,2/3)$ and
\begin{equation}\label{eq:X1f}
X_1f=\pm\frac{4}{3}f\sqrt{1-\frac{9}{4}f^2}.
\end{equation}

As $X_2f=0$, the last equation can be rewritten as

\begin{equation}\label{eq:fderiv}
\frac{f'}{f\sqrt{1-\frac{9}{4}f^2}}=\pm\frac{4}{3}.
\end{equation}
Replacing \eqref{eq:X1f} in \eqref{eq7}, by a straightforward computation, we obtain
$$
\overline{\nabla}_{X_1}\tilde{N}_2=\pm\sqrt{1-\frac{9}{4}f^2}\tilde{N}_2.
$$
Thus, as $a_1(u)=\tilde{N}_2(u)$, using \eqref{eq:fderiv} and the above relation, it follows that
$$
\frac{d a_1}{du}=\overline{\nabla}_{\sigma'}\tilde{N}_2=\frac{3}{4}\frac{f'}{f}a_1.
$$
Integrating the above equation, one gets $a_1(u)=c_3 f^{3/4}(u)$, where $c_3\in\mathbb{R}^4_1$. As $\langle a_1,a_1\rangle=0$, it follows that the constant vector $c_3$ satisfies $\langle c_3,c_3\rangle=0$. Denoting $c_1=c_3/2\in\mathbb{R}^4_1$ and $c_2=a_2\in\mathbb{R}^4_1$, the local parametrization of $M$ can be rewritten as
$$
X(u,v)=\sigma(u)+f^{3/4}(u)v^2c_1+vc_2, \qquad u\in (-\delta,\delta), v\in(-\varepsilon,\varepsilon),
$$
where $\left|c_1\right|=0$, $\left|c_2\right|=1$, $\langle c_1,c_2\rangle=0$ and
$$
\langle \sigma(u),c_1\rangle=-\frac{1}{2f^{3/4}(u)}, \qquad \langle \sigma (u),c_2\rangle=0.
$$
We note that the curve $\sigma$ is parametrized by arc-length and, as a curve in $\mathbb{H}^2=\mathbb{H}^3\cap \Pi$, where $\Pi$ is the linear hyperspace of $\mathbb{R}^{4}_{1}$ defined by $\langle \overline{r},c_2\rangle =0$, has the geodesic curvature $\kappa(u)=f(u)/2>0$, for any $u\in (-\delta,\delta)$.

Since $f(u)=2\kappa(u)$, from \eqref{eq:fderiv} one gets
$$
\kappa'(u)=\pm\frac{4}{3}\kappa(u)\sqrt{1-9\kappa^2(u)},
$$
with $\kappa(u)\in (0,1/3)$, for any $u\in (-\delta,\delta)$. Now, it is easy to see that $\kappa'\neq 0$, and the solution of the above equation is given by
$$
u(\kappa)=\pm\frac{3}{4}\log\left(\frac{\kappa}{1+\sqrt{1-9\kappa^2}}\right)+C, \qquad \kappa\in\left(0,\frac{1}{3}\right), C\in\mathbb{R}.
$$

\end{proof}


\begin{remark}
We note that the surface $\left(M^2,X^\ast\langle,\rangle\right)$ is isometric to the abstract standard biconservative surface $\left(D_{0},g_{0}\right)$, that means $C_{-1}=0$.
\end{remark}

\begin{remark}
We note that the biconservative surface defined by $X_{\tilde{C}_{-1}}$ is a ``parabola'' which lies in a $2$-affine plane parallel with the $2$-plane spanned  by $c_1$ and $c_2$, and its vertex belongs to the curve $\sigma$.
\end{remark}

Further, considering $c_1=e_1+e_4$ and $c_2=e_2$, where $\left\{e_1,e_2,e_3,e_4\right\}$ is the canonical basis of $\mathbb{R}^4_1$, the curve $\sigma$ can be rewritten as
$$
\sigma(u)=\left(y(u)-\frac{1}{2^{7/4}\kappa^{3/4}(u)},0,x(u),y(u)\right),
$$
where the functions $x=x(u)$ and $y=y(u)$ must satisfy
\begin{equation}\label{eq:system3}
\left\{
\begin{array}{l}
\left(y(u)-\frac{1}{2^{7/4}\kappa^{3/4}(u)}\right)^2+x^2(u)-y^2(u)=-1, \qquad y(u)>0 \\\\
\left(y(u)-\frac{1}{2^{7/4}\kappa^{3/4}(u)}\right)'^2+x'^2(u)-y'^2(u)=0\\\\
\left(y(u)-\frac{1}{2^{7/4}\kappa^{3/4}(u)}\right)''^2+x''^2(u)-y''^2(u)=\kappa^2(u)-1.
\end{array}
\right.
\end{equation}

From the first equation of \eqref{eq:system3}, we obtain
$$
y(u)=2^{3/4}\kappa^{3/4}(u)x^2(u)+2^{3/4}\kappa^{3/4}(u)+\frac{1}{2^{11/4}\kappa^{3/4}(u)}.
$$
Replacing the first derivative of $y$ respect to $u$ from above equation in the second equation of \eqref{eq:system3} and using  \eqref{eq:firstIntegralW=0}, we get
$$
\left(x'(u)+\sqrt{1-9\kappa^2(u)} x(u)\right)^2=9\kappa^2(u).
$$
As $\kappa=\kappa(u)>0$, for any $u$, we can consider $u=u(\kappa)$, so $x=x(u(\kappa))=x(\kappa)$, and if we denote the derivative with respect to $\kappa$ by ``$\cdot$'', we get
$$
\left(\frac{4}{3}\kappa\dot{x}(\kappa)+x(\kappa)\right)^2=\frac{9\kappa^2}{1-9\kappa^2}.
$$
From this equation, it follows that
$$
\dot{x}(\kappa)=-\frac{3}{4\kappa}x(\kappa)\pm\frac{9}{4\sqrt{1-9\kappa^2}}
$$
and, one obtains
\begin{equation}\label{eq:x}
x(\kappa)=\frac{1}{\kappa^{3/4}}\mu(\kappa), \qquad \kappa\in\left(0,\frac{1}{3}\right),
\end{equation}
where
\begin{equation}\label{eq:muk3}
\mu(\kappa)=\pm \frac{9}{4}\int_{\kappa_{00}}^{\kappa}\frac{\tau^{3/4}}{\sqrt{1-9\tau^2}}\ d\tau+c_0,
\end{equation}
with $c_0\in\mathbb{R}$. If we denote by
$$
\mu_0(\kappa)=\frac{9}{4}\int_{\kappa_{00}}^{\kappa}\frac{\tau^{3/4}}{\sqrt{1-9\tau^2}}\ d\tau, \qquad \kappa\in  \left(0,\frac{1}{3}\right),
$$
where $\kappa_{00}$ is arbitrarily fixed in $\left(0,1/3\right)$, then we can write $\mu(\kappa)=\pm \mu_{0}(\kappa)+c_0$.

Now, it is easy to verify that the last equation of \eqref{eq:system3} is automatically satisfied.

Next, we can write the explicit parametrization of the curve $\sigma$
\begin{align*}
\sigma(\kappa)= & \left(y(\kappa)-\frac{1}{2^{7/4}\kappa^{3/4}},0,x(\kappa),y(\kappa)\right) \\
  \equiv & \left(y(\kappa)-\frac{1}{2^{7/4}\kappa^{3/4}},x(\kappa),y(\kappa)\right), \qquad \kappa\in\left(0,\frac{1}{3}\right)
\end{align*}
where the function $x=x(\kappa)$ is given in \eqref{eq:x} and $y=y(\kappa)$ is defined as
$$
y(\kappa)=2^{3/4}\kappa^{3/4}\left(x^2(\kappa)+1\right)+\frac{1}{2^{11/4}\kappa^{3/4}}.
$$
Thus,
\begin{eqnarray}\label{eq:localSurface3} \nonumber
  X(\kappa,v) &=&  \left(2^{3/4}\kappa^{3/4}\left(1+x^2(\kappa)+v^2\right)-\frac{1}{2^{11/4}\kappa^{3/4}},v,x(\kappa),\right. \\
   & &  \left.\quad  2^{3/4}\kappa^{3/4}\left(1+x^2(\kappa)+v^2\right)+\frac{1}{2^{11/4}\kappa^{3/4}}\right),
\end{eqnarray}
for any $(\kappa,v)\in \left(0,1/3\right)\times\mathbb{R}$.

Through the standard diffeomorphism from hyperbolid model to upper half space model, the ``profile curve'' $\sigma$ becomes
$$
\sigma(\kappa)=\frac{\left(\mu(\kappa),\kappa^{3/4}\right)}{2^{3/4}\left(\kappa^{3/2}+\mu^2(\kappa)\right)}, \qquad \kappa\in \left(0,\frac{1}{3}\right).
$$

Further, we prove that the limits of $\mu_0$ as $\kappa$ approaches $0$ and $1/3$ are finite.

In order to show that the limit of $\mu_0$ as $\kappa$ approaches $0$ is finite, first, let us consider the change of variables $\tilde{\tau}=1/\tau$. The integral $\mu_0$ becomes
$$
\mu_0(\kappa)=\int_{1/\kappa}^{1/\kappa_{00}}\frac{9}{4\tilde{\tau}^{7/4}\sqrt{\tilde{\tau}^2-9}}\ d\tilde{\tau}.
$$
Since
$$
\lim_{\tilde{\tau}\to\infty} \tilde{\tau}^{7/4} \cdot \frac{9}{4\tilde{\tau}^{7/4}\sqrt{\tilde{\tau}^2-9}}=0\in [0,\infty),
$$
we get that the integral
$$
\int_{1/\kappa_{00}}^{\infty} \frac{9}{4\tilde{\tau}^{7/4}\sqrt{\tilde{\tau}^2-9}}\ d\tilde{\tau}<\infty,
$$
so we obtain
$$
\lim_{\kappa\searrow 0}\mu_{0}(\kappa)=\mu_{0,-1}\in (-\infty,0).
$$
In order to show that the limit of $\mu_0$ as $\kappa$ approaches $1/3$ is finite, we note that
$$
\lim_{\kappa\nearrow 1/3} \sqrt{\frac{1}{3}-\tau} \frac{9\tau^{3/4}}{4\sqrt{1-9\tau^2}}=\frac{3^{3/4}}{2^{5/2}}\in[0,\infty),
$$
thus,
$$
\int_{\kappa_{00}}^{1/3}\frac{9\tau^{3/4}}{4\sqrt{1-9\tau^2}}\ d\tau<\infty,
$$
and then
$$
\lim_{\kappa\nearrow 1/3}\mu_{0}(\kappa)=\mu_{0,1}\in (0,\infty).
$$
As in the previous cases, our aim is to find two profile curves of the standard surfaces and glue them at least of class $C^3$ in order to obtain a closed regular curve in the upper half plane, which would define the complete biconservative surface in $\mathbb{H}^3$.

Considering the curves $\sigma_1$ and $\sigma_2$ given by
$$
\sigma_1(\kappa)=\frac{\left(\mu_0(\kappa),\kappa^{3/4}\right)}{2^{3/4}\left(\kappa^{3/2}+\mu_0^2(\kappa)\right)},\quad
\sigma_2(\kappa)=\frac{\left(2\mu_{0,1}-\mu_0(\kappa),\kappa^{3/4}\right)}{2^{3/4}\left(\kappa^{3/2}+\left(2\mu_{0,1}-\mu_0(\kappa)\right)^2\right)},
$$
that means we choose $c_0=0$ and the sign ``+'', and $c_0=2\mu_{0,1}$ and the sign ``-'', respectively in \eqref{eq:muk3}, by a direct computation we can see that gluing the curves $\sigma_1$ and $\sigma_2$, we obtain at least a $C^3$ smooth closed curve in the upper half plane.

We note that the immersion $X$ is in fact an embedding and the profile curve has no self-intersections. Therefore, we get a closed biconservative surface in $\mathbb{H}^3$, which has to be complete. In Figure $5$ we represent the curves $\sigma_1$  and $\sigma_2$, with the colors red and blue, respectively, and in Figure $6$ we represent the corresponding surfaces to $\sigma_1$  and $\sigma_2$ in the upper half space (with the Euclidean metric).

\begin{tabular}{c c}
\includegraphics[width=0.5\textwidth]{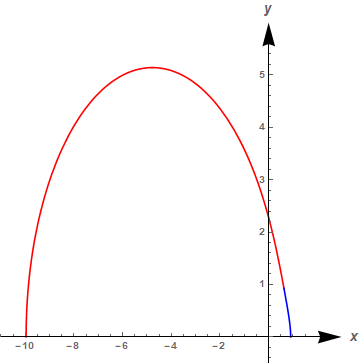}
&
\includegraphics[width=0.5\textwidth]{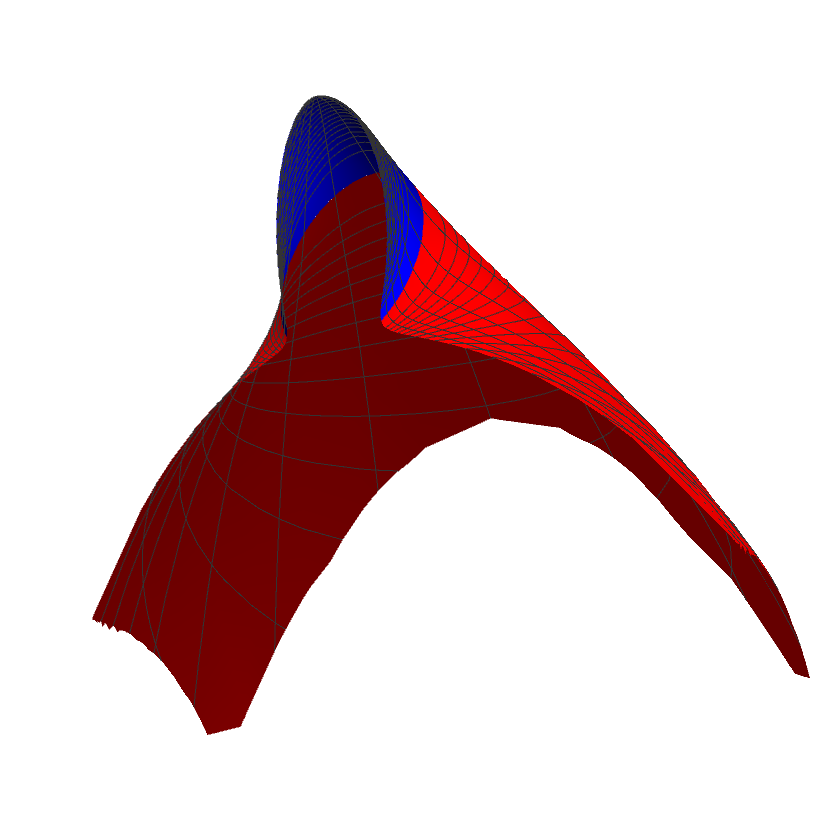}
\\
Figure $5$. The profile curves
&
Figure $6$. The corresponding surfaces
\\
$\sigma_1$  and $\sigma_2$
&
to $\sigma_1$  and $\sigma_2$
\end{tabular}

We can conclude with our last theorem.

\begin{theorem}\label{th-complete-extrinsic}
By gluing two standard biconservative surfaces along their common boundary we get a complete biconservative regular surface in $\mathbb{H}^3$. Moreover, the gradient of its mean curvature vanishes along the initial boundary which now is a geodesic of the surface.
\end{theorem}


\begin{thebibliography}{99}
\bibitem{BMO13} A.~Balmu\c s, S.~Montaldo, C.~Oniciuc, \textit{Biharmonic PNMC submanifolds in spheres}, Ark. Mat. 51 (2013), 197--221.
\bibitem{CMOP14} R.~Caddeo, S.~Montaldo, C.~Oniciuc, P.~Piu, \textit{Surfaces in three-dimensional space forms with divergence-free stress-bienergy tensor}, Ann. Mat. Pura Appl. (4) 193 (2014), 529--550.
\bibitem{FNO16} D.~Fetcu, S.~Nistor, C.~Oniciuc, \textit{On biconservative surfaces in $3$-dimensional space forms}, Comm. Anal. Geom. (5) 24 (2016), 1027--1045.
\bibitem{FOP15} D.~Fetcu, C.~Oniciuc, A.L.~Pinheiro, \textit{CMC biconservative surfaces in $\mathbb{S}^n\times\mathbb{R}$ and $\mathbb{H}^n\times\mathbb{R}$}, J. Math. Anal. Appl. 425 (2015), 588--609.
\bibitem{F15} Y.~Fu, \textit{Explicit classification of biconservative surfaces in Lorentz $3$-space forms}, Ann. Mat. Pura Appl. (4) 194 (2015), 805--822.
\bibitem{FT16} Y.~Fu, N.C.~Turgay, \textit{Complete classification of biconservative hypersurfaces with diagonalizable shape operator in Minkowski 4-space}, Internat. J. Math. (5) 27 (2016), 1650041, 17 pp.
\bibitem {G73} W.B.~Gordon, \textit{An analytical criterion for the completeness of Riemannian manifolds}, Proc. Amer. Math. Soc. 37 (1973), 221--225.
\bibitem{HV95} Th.~Hasanis, Th.~Vlachos, \textit{Hypersurfaces in $E^4$ with harmonic mean curvature vector field}, Math. Nachr. 172 (1995), 145--169.
\bibitem{J86} G.Y.~Jiang, \textit{$2$-harmonic maps and their first and second variational formulas}, Chinese Ann. Math. Ser. A 7 (1986), no. 4, 389--402.
\bibitem{J87} G.Y.~Jiang, \textit{The conservation law for $2$-harmonic maps between Riemannian manifolds}, Acta Math. Sinica 30 (1987), no. 2, 220--225.
\bibitem{LMO08} E.~Loubeau, S.~Montaldo, C.~Oniciuc, \textit{The stress-energy tensor for biharmonic maps}, Math. Z. 259 (2008), 503--524.
\bibitem{MOR16} S.~Montaldo, C.~Oniciuc, A.~Ratto, \textit{Proper biconservative immersions in the Euclidean space}, Ann. Mat. Pura Appl. (4) 195 (2016), 403--422.
\bibitem{MM15} A.~Moroianu, S.~Moroianu, \textit{Ricci surfaces}, Ann. Sc. Norm. Super. Pisa Cl. Sci. (5) 14 (2015), 1093--1118.
\bibitem{NPhD17} S.~Nistor, \textit{Biharmonicity and biconservativity topics in the theory of submanifolds}, PhD Thesis, 2017.
\bibitem{N16} S.~Nistor, \textit{Complete biconservative surfaces in $\mathbb{R}^3$ and $\mathbb{S}^3$}, J. Geom. Phys. 110 (2016), 130--153.
\bibitem{NO19} S.~Nistor, C.~Oniciuc, \textit{On the uniqueness of complete biconservative surfaces in $\mathbb{R}^3$}, Proc. Amer. Math. Soc. (3) 147 (2019), 1231--1245.
\bibitem{NO17} S.~Nistor, C.~Oniciuc, \textit{Global properties of biconservative surfaces in $\mathbb{R}^3$ and $\mathbb{S}^3$}, Proceedings of The International Workshop on Theory of Submanifolds, Istanbul, Turkey, vol. 1 (2016), 30--56.
\bibitem{O02} C.~Oniciuc, \textit{Biharmonic maps between Riemannian manifolds}, An. Stiint. Univ. Al.I. Cuza Iasi Mat (N.S.) 48 (2002), 237--248.
\bibitem{O10} Y.L.~Ou, \textit{Biharmonic hypersurfaces in Riemannian manifolds}, Pacific J. Math. 248 (2010), 217--232.
\bibitem{S12} T.~Sasahara, \textit{Surfaces in Euclidean 3-space whose normal bundles are tangentially biharmonic}, Arch. Math. (Basel) (3) 99 (2012), 281--287.
\bibitem{S15} T.~Sasahara, \textit{Tangentially biharmonic Lagrangian H-umbilical submanifolds in complex space forms}, Abh. Math. Semin. Univ. Hambg. 85 (2015), 107--123.
\bibitem{T15} N.C.~Turgay, \textit{$H$-hypersurfaces with three distinct prinicipal curvatures in the Euclidean spaces}, Ann. Math. 194 (2015), 1795--1807.
\bibitem{UT16} A.~Upadhyay, N.C.~Turgay, \textit{A Classification of Biconservative Hypersurfaces in a Pseudo-Euclidean Space}, J. Math. Anal. Appl. (2) 444 (2016), 1703--1720.
\bibitem{YT18} R.~Yeğin Şen, N.C.~Turgay, \textit{On biconservative surfaces in 4-dimensional Euclidean space}, J. Math. Anal. Appl. 460 (2018), no. 2, 565--581.

\end{thebibliography}
\end{document}